\documentclass[11pt]{amsart}

\usepackage{amsmath,amssymb,latexsym,soul,cite,mathrsfs}

\usepackage{color,enumitem,graphicx}
\usepackage[colorlinks=true,urlcolor=blue,
citecolor=red,linkcolor=blue,linktocpage,pdfpagelabels,
bookmarksnumbered,bookmarksopen]{hyperref}
\usepackage[english]{babel}

\usepackage[left=2.9cm,right=2.9cm,top=2.8cm,bottom=2.8cm]{geometry}
\usepackage[hyperpageref]{backref}

\usepackage[colorinlistoftodos]{todonotes}
\makeatletter
\providecommand\@dotsep{5}
\def\listtodoname{List of Todos}
\def\listoftodos{\@starttoc{tdo}\listtodoname}
\makeatother

\numberwithin{equation}{section}

\newtheorem{theorem}{Theorem}[section]
\newtheorem{proposition}[theorem]{Proposition}
\newtheorem{lemma}[theorem]{Lemma}
\newtheorem{corollary}[theorem]{Corollary}

\newcommand\R{\mathbb R}
\newcommand\N{\mathbb N}

\begin{document}

\title[exterior domains with nonlocal Neumann boudary condition]
{Fractional elliptic problem in exterior domains with nonlocal Neumann condition}

\author{Claudianor O. Alves}
\author{C\'esar E. Torres Ledesma}

\address[Claudianor O. Alves]{\newline\indent Unidade Acad\^emica de Matem\'atica
\newline\indent 
Universidade Federal de Campina Grande,
\newline\indent
58429-970, Campina Grande - PB - Brazil}
\email{\href{mailto:coalves@dme.ufcg.edu.br}{coalves@dme.ufcg.edu.br}}

\address[C\'esar E. Torres Ledesma]
{\newline\indent Departamento de Matem\'aticas
\newline\indent 
Universidad Nacional de Trujillo
\newline\indent
Av. Juan Pablo II s/n. Trujillo-Per\'u}
\email{\href{ctl\_576@yahoo.es}{ctl\_576@yahoo.es}}

\pretolerance10000


\begin{abstract}
\noindent In this paper we consider the existence of solution for the following class of fractional elliptic problem 
\begin{equation}\label{00}
\left\{\begin{aligned}
(-\Delta)^su + u &= Q(x) |u|^{p-1}u\;\;\mbox{in}\;\;\R^N \setminus \Omega\\
\mathcal{N}_su(x) &= 0\;\;\mbox{in}\;\;{\Omega},
\end{aligned}
\right.
\end{equation} 
where $s\in (0,1)$, $N> 2s$, $\Omega\subset \R^N$ is a bounded set with smooth boundary, $(-\Delta)^s$ denotes the fractional Laplacian operator and $\mathcal{N}_s$ is the nonlocal operator that describes the Neumann boundary condition, which is given by 
$$
\mathcal{N}_su(x) = C_{N,s} \int_{\R^N \setminus \Omega} \frac{u(x) - u(y)}{|x-y|^{N+2s}}dy,\;\;x\in {\Omega}. 
$$    

\end{abstract}

\thanks{ C. O. Alves was partially 	supported by  CNPq/Brazil 304804/2017-7  and C.E. Torres Ledesma was partially supported by INC Matem\'atica  88887.136371/2017.}

\subjclass[2010]{Primary 35A15; Secondary 35J60, 35C15.} 
\keywords{Variational Methods, Nonlinear elliptic equations, Integral representations of solutions}

\maketitle

\section{Introduction}

In this paper, we consider the existence of weak solution for the following class of fractional elliptic problem with nonlocal Neumann conditions: 
$$
\left\{
\begin{aligned}
(-\Delta)^su + u &= Q(x)|u|^{p-1}u\;\;\mbox{in}\;\;\R^N \setminus \Omega\\
\mathcal{N}_su(x) &= 0\;\;\mbox{in}\;\;\Omega,
\end{aligned}
\right.
\eqno{(P)}
$$
where $s\in (0,1)$, $N> 2s$, $p\in (1, \frac{N+2s}{N-2s})$, $\Omega$ is a smooth bounded domain of $\R^N$, and $(-\Delta)^s$ denotes the fractional Laplacian operator defined as,
\begin{equation}\label{I01}
(-\Delta)^su(x) = C_{N,s} P.V.\int_{\R^N}\frac{u(x)-u(y)}{|x-y|^{N+2s}}dy.
\end{equation}
By $\mathcal{N}_s$ we denote the nonlocal normal derivative, defined as   
\begin{equation}\label{I02}
\mathcal{N}_su(x) = C_{N,s} \int_{\R^N \setminus \Omega} \frac{u(x) - u(y)}{|x-y|^{N+2s}}dy,\;\;x\in {\Omega}.
\end{equation}
This function was introduced by Dipierro et al. \cite{SDXREV}, where the authors proved that, when $s\to 1^-$, the classical Neumann boundary condition $\frac{\partial u}{\partial \eta}$ is recovered in some sense. Hereafter, without loss of generality we will assume that $C_{N,s}=1$.

Hereafter, $Q$ is a continuous function such that 
\begin{enumerate}
	\item[$(Q_1)$] $Q(x)\geq \tilde{Q} >0$ in $\R^N \setminus \Omega$ and 
	$$\lim_{|x|\to +\infty} Q(x) = \tilde{Q}.$$
\end{enumerate}

When $s=1$, the problem $(P)$ reduces to the elliptic problem 
\begin{equation}\label{I03}
\left\{
\begin{aligned}
-\Delta u + u &= Q(x)|u|^{p-1}u,\;\;\mbox{in}\;\;\R^N \setminus \Omega\\
\frac{\partial u}{\partial \eta} &= 0\;\;\mbox{on}\;\;\partial \Omega.
\end{aligned}
\right.
\end{equation}

In \cite{VBGC}, Benci and Cerami showed that (\ref{I03}), with $Q\equiv 1$ and Dirichlet condition, does not have a ground state solution, that is, a solution of (\ref{I03}) with minima energy. However, Esteban in \cite{ME} proved that the same problem with Neumann condition has a ground state solution. Furthermore, in \cite{DC}, Cao studied the existence of positive solution for problem (\ref{I03}) by supposing that 
\begin{enumerate}
\item[$(Q'_1)$] $Q(x)\geq \tilde{Q}-Ce^{-\nu|x|}|x|^{-m}$ \quad as \quad $|x| \to +\infty$ and $\displaystyle \lim_{|x|\to +\infty}Q(x) = \tilde{Q}>0$,
\end{enumerate}
where $\nu=\frac{2(p+1)}{p-1}$, $m >N-1$ and $C>0$. In the same paper, Cao also studied the existence of solution that changes sign ( nodal solution ), by assuming the following additional condition on $Q$
\begin{enumerate}
	\item[$(Q'_2)$] $Q(x)\geq \tilde{Q}+Ce^{-\frac{p|x|}{p+1}}|x|^{-m}$ \quad as \quad $|x| \to +\infty$ and $\displaystyle \lim_{|x|\to +\infty}Q(x) = \tilde{Q}>0$.
\end{enumerate}
with $0<m < \frac{N-1}{2}$. In \cite{CAPCEM}, Alves et al. showed that the results found in \cite{DC} also hold for the $p$-Laplacian operator and for a larger class of nonlinearity. We also mention the work by Alves \cite{CA}, where problem (\ref{I03}) was considered with critical growth nonlinearity for $N=2$.  It is a very important point out that in all the above mentioned papers the fact that the limit problem in whole $\mathbb{R}^N$ has a ground state solution with exponential decaying is a key point in their arguments, because this type of behavior at infinite works well with conditions $(Q'_1)$ and $(Q'_2)$.

Recently, the case $s \in (0,1)$ has received a special attention, because involves the fractional Laplacian operator $(-\Delta)^{s}$, which arises in a quite natural way in many different contexts, such as, among the others, the thin obstacle problem, optimization, finance, phase transitions, stratified materials, anomalous diffusion, crystal dislocation, soft thin films, semipermeable membranes, flame propagation, conservation laws, ultra-relativistic limits of quantum mechanics, quasi-geostrophic flows, multiple scattering, minimal surfaces, materials science and water waves, for more detail see \cite{Bucurb, EDNGPEV, Dipierrob, Molicab, CP}.

In the last 20 years, there has been a lot of interest in the study of the existence and multiplicity of nodal solutions for nonlinear elliptic problems. There are some powerful methods which have been developed, such as the descended flow methods \cite{Liu01}, constrained minimization methods \cite{Bartsch05}, super and sub solution combining with truncation techniques \cite{Dancer95} and so on. Recently, the existence and multiplicity of nodal solutions for the fractional elliptic problem 
\begin{equation}\label{I04}
\left\{
\begin{aligned}
(-\Delta)^su &= f(x,u)\;\;\mbox{in}\;\;\Omega,\\
u&=0\;\;\mbox{in}\;\;\R^N \setminus \Omega,
\end{aligned}
\right.
\end{equation}
where $s\in (0,1)$ and $\Omega\subset \R^N$ is a smooth bounded domain, has been investigated by Chang and Wang \cite{Chang14}, by using the descended flow methods and harmonic extension techniques. Teng et al. \cite{KTKWRW} have prove the existence of nodal solutions for problem (\ref{I04}) by using the constrained minimization methods and adapting some arguments found in \cite{AlvesSouto}. We note that the main difficulties in the study of problem (\ref{I04}) is the presence of the fractional Laplacian $(-\Delta)^s$ that is a nonlocal operator. Indeed, the Euler-Lagrange functional associated to the problem (\ref{I04}), that is
$$
J(u) = \frac{1}{2}\iint_{\R^{2N}\setminus (\Omega^c\times \Omega^c)}\frac{|u(x) - u(y)|^2}{|x-y|^{N+2s}}dy dx - \int_{\Omega}F(x,u(x))dx 
$$
does not satisfy the decompositions 
$$
\begin{aligned}
&J(u) = J(u^+) + J(u^-)\\
&J'(u)u^{\pm} = J'(u^{\pm})u^{\pm},
\end{aligned}
$$
which were fundamental in the application of variational methods to study (\ref{I04}); see \cite{BartschW05}. We also mention a recent work by Ambrosio and Isernia \cite{VATI}, where the fractional Schr\"odinger equation with vanishing potentials 
\begin{equation}\label{I05} 
(-\Delta)^s u + V(x) u = K(x)f(u)\;\;\mbox{in}\;\;\R^N
\end{equation}
was studied. By using a minimization argument and a quantitative deformation Lemma, the authors proved the existence of nodal solutions for (\ref{I05}).

On the other hand, research has been done in recent years for the fractional elliptic problem with nonlocal Neumann condition. We mention the work by Dipierro et al. \cite{SDXREV}, where they established a complete description of the eigenvalues of $(-\Delta)^s$ with zero nonlocal Neumann boundary condition, an existence and uniqueness result for the elliptic problem and the main properties of the fractional heat equation with this type of boundary condition. Chen \cite{GCh}, has considered the fractional Schr\"odinger equation
\begin{equation}\label{I06}
\left\{
\begin{aligned}
\epsilon^{2s}(-\Delta)^su + u &= |u|^{p-1}u\;\;\mbox{in}\;\;\Omega,\\
\mathcal{N}_su &= 0\;\;\mbox{on}\;\;\R^N \setminus \overline{\Omega},
\end{aligned}
\right.
\end{equation}
where $\epsilon >0$, $s\in (0,1)$, $\Omega\subset \R^N$ be a smooth bounded domain, $p\in (1, \frac{N+2s}{N-2s})$ and 
$$
\mathcal{N}_su(x) = \int_{\Omega}\frac{u(x) - u(y)}{|x-y|^{N+2s}}dy,\;\;x\in \R^N \setminus \overline{\Omega}.
$$
By using mountain pass theorem, he showed that there exists a non-negative solution $u_\epsilon$ to (\ref{I06}). For further results with mixed boundary condition see  \cite{Barrios17} and \cite{Leonori17}.

Motivated by the previous works, and by the fact that after a bibliography review we did not find in the literature any paper dealing with $(P)$ in exterior domains and Neumann boundary condition, the present paper concerns with the existence of two nontrivial solutions for problem $(P)$, the first solution is a non-negative ground state solution while the second one is a nodal solution. However, different of  the local case $s=1$, we do not know if the  ground state solution of limit problem in whole $\mathbb{R}^N$ has an exponential decaying, which brings a lot of difficulties for the nonlocal case. The reader is invited to see that for the existence of nodal solution,  we overcome this difficulty by assuming more a condition on the function $Q$, see Theorem \ref{nodal} below. Moreover, we prove a Lions type theorem for exterior domain that is crucial in our approach, see Proposition \ref{GC-C} in Section 3.  The main results of this paper, in some sense, complete the study made in \cite{DC}, because we are considering a version of that paper for the fractional Laplacian. The reader is invited to see that we were not able to work with conditions like $(Q'_1)$ and $(Q'_2)$, because in our case we do not know if the ground state solution of the limit problem has an exponential decay at infinite. Finally, we would like point out that in \cite{CACT}, Alves et al. have studied problem $(P)$ in an exterior domain with Dirichlet boundary conditions and proved that the problem does not have ground state solution.

Concerning the existence of a non-negative ground state solution, we has the following result.
\begin{theorem}\label{ground}
Suppose that $p\in (1, \frac{N+2s}{N-2s})$ and $(Q_1)$ holds. Then $(P)$ has a ground state solution.
\end{theorem} 
The proof of Theorem \ref{ground} is obtained by adapting some arguments developed in \cite{CAPCEM}, \cite{CACT} and \cite{SDXREV}. More precisely, we will find critical points of the functional $I: H_{\tilde{\Omega}}^{s} \to \R$ associated to $(P)$ given by   
\begin{equation} \label{FEN}
\begin{aligned}
I(u) &= \frac{1}{2}\left(\frac{1}{2}\iint_{\R^{2N}\setminus \Omega^2}\frac{|u(x) - u(y)|^2}{|x-y|^{N+2s}}dy dx +\int_{\R^N\setminus \Omega}|u|^2dx\right) - \frac{1}{p+1}\int_{\R^N \setminus \Omega} Q(x)|u|^{p+1}dx\\
\end{aligned}
\end{equation}
on the Nehari manifold 
$$
\mathcal{N} = \{u\in H_{\tilde{\Omega}}^{s}\setminus \{0\}:\;\;I'(u)u = 0\},
$$
where $\tilde{\Omega} = \R^N \setminus \Omega$ and $H_{\tilde{\Omega}}^{s}$ is the Sobolev space given by 
$$
H_{\tilde{\Omega}}^{s} = \left\{u: \R^N \to \R\quad \mbox{measurable}\;:\;\frac{1}{2}\iint_{\R^{2N} \setminus \Omega^2} \frac{|u(x) - u(y)|^2}{|x-y|^{N+2s}}dy dx + \int_{\mathbb{R}^N \setminus \Omega} |u|^2 dx<\infty \right\},
$$
endowed with the norm
$$
\|u\|_{H_{\tilde{\Omega}}^{s}}=\left(\frac{1}{2} \iint_{\R^{2N} \setminus \Omega^2} \frac{|u(x) - u(u)|^2}{|x-y|^{N+2s}}dy dx + \int_{\Omega} |u|^2 dx\right)^{\frac{1}{2}}.
$$

Our second main result is concerned to existence of nodal solution, and in this case we fix $D>0$ such that  
$$
|x| \leq D, \quad  \forall x \in \Omega.
$$

\begin{theorem}\label{nodal}
Suppose $p\in (1, \frac{N+2s}{N-2s})$, $(Q_1)$ and  that there are $C>0$, $\gamma > (N+2s)(p+1) - N$, $R>D+1$ and $\sigma_R \in \mathbb{R}^N$ with $|\sigma_R| > 3R$ such that 
$$
\displaystyle Q(x) - \tilde{Q} \geq CR^{\gamma},\quad \forall x\in B(\sigma_R,2R) \setminus B(\sigma_R,R). \leqno{(Q_2)}
$$ 
Then, there is $R_0>0$ such that $(P)$ has a nodal solution for all $R \geq R_0$.
\end{theorem} 

Before concluding this introduction, we would like point out that our results are true for a large class of nonlinearity $f$, however we decided to work with the case $f(t)=|t|^{p-2}t$ to become the ideas more clear to the reader. 

The paper is organized as follows: In Section 2, we prove some results involving the limit problem. In Section 3, we prove Theorem \ref{ground},  while in Section 4 we show Theorem \ref{nodal}. Finally, in Section 5, we wrote an appendix related to the regularity and behavior at infinite of the ground state solution obtained in Theorem \ref{ground}, because these facts are very important in our paper. 

\section{Preliminary Results}

In this section we introduce some function spaces and consider the existence of positive solution of the limit problem 
$$
\left\{
\begin{aligned}
(-\Delta)^su + u &= \tilde{Q}|u|^{p-1}u\;\;\mbox{in}\;\;\R^N, \\
u&\in H^s(\R^N).
\end{aligned}
\right.
\eqno{(P_\infty)}
$$
We denote by $H^s(\R^N)$ the fractional Sobolev spaces endowed with the norm 
$$
\|u\|_s = \left(\frac{1}{2}\int_{\R^N}\int_{\R^N} \frac{|u(x) - u(y)|^2}{|x-y|^{N+2s}}dy dx + \int_{\R^N}|u|^2dx\right)^{\frac{1}{2}}.
$$

\noindent 
If  $G \subset \R^N$ is a smooth domain, we introduce the fractional space
$$
H_{G}^{s} = \left\{u: \R^N \to \R\quad \mbox{measurable}\;:\;\frac{1}{2}\iint_{\R^{2N} \setminus (G^c)^2} \frac{|u(x) - u(u)|^2}{|x-y|^{N+2s}}dy dx + \int_{G}|u|^2dx <\infty\right\},
$$
endowed with the norm
$$
\|u\|_{H_{G}^{s}}= \left(\frac{1}{2}\iint_{\R^{2N} \setminus (G^c)^2} \frac{|u(x) - u(u)|^2}{|x-y|^{N+2s}}dy dx + \int_{G}|u|^2dx\right)^{\frac{1}{2}}.
$$
It is well known that $H_{G}^{s}$ is a Hilbert space with the inner product $\langle\cdot, \cdot \rangle_{H_{G}^{s}}$, given by
$$
\langle u, v \rangle_{H_{G}^{s}} = \frac{1}{2}\iint_{\R^{2N}\setminus (G^c)^2}\frac{[u(x) - u(y)][v(x) - v(y)]}{|x-y|^{N+2s}}dy dx + \int_{G} u(x)v(x)dx.
$$ 

Related to the $H_{G}^{s}$ we have important information that is stated in lemma below, which can be found in \cite{FDGD}.

\begin{lemma} \label{Pnta01}
\begin{enumerate} 
\item Let $H^s(G)$ the classical fractional Sobolev space endowed with the norm 
$$
\|u\|_{H^s(G)}^{2} = \frac{1}{2}\int_{G}\int_{G} \frac{|u(x) - u(y)|^2}{|x-y|^{N+2s}}dy dx + \int_{G}|u|^2dx.
$$
Since $G \times G \subset \R^{2N} \setminus (G^c \times G^c)$, then the embedding $H_{G}^{s} \hookrightarrow H^s(G)$ is continuous. 
\item The embedding $H^s(\R^N) \hookrightarrow H_{G}^{s}$ is continuous.
\item Since $H^s(G) \hookrightarrow L^p(G)$ continuously for every $p\in [2, \frac{2N}{N-2s}]$, by $(1)$ we have 
$$
H_G^s \hookrightarrow L^p(G) \;\;\mbox{for all}\;\;p\in \left[2, \frac{2N}{N-2s}\right].
$$
\item If $G$ is bounded, we have the compactness embedding 
$$
H_G^s \hookrightarrow L^p(G) \;\;\mbox{for all}\;\;p\in \left[1, \frac{2N}{N-2s}\right).
$$
\end{enumerate}
\end{lemma}

\noindent 
Associated to problem $(P_\infty)$, we have the functional $I_\infty : H^s(\R^N) \to \R$ defined as 
\begin{equation}\label{P01}
I_\infty (u) = \frac{1}{2}\left(\frac{1}{2}\iint_{\R^{2N}}\frac{|u(x) - u(y)|^2}{|x-y|^{N+2s}}dy dx +\int_{\R^N}|u|^2dx\right) - \frac{1}{p+1}\int_{\R^N}\tilde{Q}|u|^{p+1}dx.
\end{equation}
It is standard to show that $I_\infty \in C^1(H^s(\R^N), \R)$ with  
\begin{equation}\label{P02}
I'_\infty(u)v = \frac{1}{2}\iint_{\R^{2N}}\frac{[u(x) - u(y)][v(x) - v(y)]}{|x-y|^{N+2s}}dy dx + \int_{\R^N}uvdx - \int_{\R^N} \tilde{Q}|u|^{p-2}uvdx,
\end{equation}
for all $u,v \in H^s(\R^N)$. 

We start our analysis recalling that $I_{\infty}$ satisfies the mountain pass geometry 
\begin{lemma}\label{Plm02}
The functional $I_{\infty}$ satisfies the following conditions:
\begin{enumerate}
\item[(i)] There exist $\beta, \delta>0$, such that $I_{\infty}(u) \geq \beta$ if $\|u\|_{s} = \delta$.
\item[(ii)] There exists $e\in H^{s}(\mathbb{R}^N)$ with $\|e\|_{s} > \delta$ such that $I_{\infty}(e) <0$.
\end{enumerate}
\end{lemma}

\noindent 
Let $\Gamma_{\infty} = \{\gamma \in C([0,1], H^{s}(\mathbb{R}^N)):\;\;\gamma (0) = 0, I_{\infty}(\gamma (1))<0\}$, from Lemma \ref{Plm02}, the mountain pass level
$$
c_\infty = \inf_{\gamma \in \Gamma_{\infty}} \sup_{t\in [0,1]} I_{\infty}(\gamma (t)) \geq \beta >0,
$$
is well defined, and the equality below holds 
\begin{equation}\label{P03*}
c_{\infty}= \inf_{u\in \mathcal{N}_{\infty}} I_{\infty}(u)
\end{equation}
where 
$$
\mathcal{N}_{\infty} = \{u\in H^{s}(\mathbb{R}^N)\setminus \{0\}:\;\;I'_{\infty}(u)u=0\}
$$ 
is the Nehari manifold associated to $(P_\infty)$. 

Arguing as in \cite[Theorem 1.5]{PFAQJT}( see also  \cite[Proposition 3.1]{FLS}), it is easy to prove the existence of a ground state solution $\tilde{u} \in H^{s}(\mathbb{R}^N)$, which can be chosen positive and 
\begin{equation} \label{decaimento}
0<\frac{C_1}{|x|^{N+2s}} \leq \tilde{u}(x) \leq \frac{C_2}{|x|^{N+2s}}, \quad \mbox{for all} \quad |x| \geq 1.  
\end{equation}
We recall that by a ground state we understand by a function $\tilde{u} \in H^{s}(\mathbb{R}^N)$ satisfying 
$$
I_\infty(\tilde{u})=c_\infty \quad \mbox{and} \quad I'_\infty(\tilde{u})=0.
$$

\section{Proof of Theorem \ref{ground}}

In this section, we are going to prove Theorem \ref{ground}. We start our analysis by proving a version of a Lions type lemma that is crucial in our approach.  

\begin{proposition}\label{GC-C}
Let $G \subset \mathbb{R}^N$ be an exterior domain with smooth bounded boundary and  $(u_n) \subset H_{G}^{s}$ be a bounded sequence such that 
\begin{equation}\label{GS01}
\lim_{n\to \infty}\sup_{y\in \R^N}\int_{U(y,T)} |u_n|^2dx =0,
\end{equation}
for some $T>0$ and $U(y,T) = B(y,T)\cap G$ with $U(y,T)\neq \emptyset$. Then,
\begin{equation}\label{GS02}
\lim_{n\to \infty}\int_{G}|u_n|^{p}dx = 0\;\;\mbox{for all}\;\;p\in (2,2_{s}^{*}).
\end{equation}
\end{proposition}

Before proving the above proposition, we would like to point out some facts involving our proof. When $s=1$, $H_{G}^{s} = H^1(G)$, and in this case, it is well known that the constant associated with the embedding
$$
H^1(U(y,T)) \hookrightarrow L^p(U(y,T)),\;\;\mbox{for}\;\;p\in [2, 2^{*}],
$$
does not depend of $U(y, R)$, since $U(y,R)$ verifies the uniform cone condition, see \cite{Adams1} for more details. This fact plays an important role in the proof of Proposition \ref{GC-C}. Unfortunately, after a bibliography review, we did not find any paper or book with a similar results for the fractional case, that is, $s \in (0,1)$. Here, we are going to prove Proposition \ref{GC-C} by using a new approach. 

\begin{proof} ( Proposition \ref{GC-C} )
Since $G$ is an exterior domain with smooth boundary, then $\Omega = \R^N \setminus G$ is a smooth bounded domain. Moreover, by using extension operator $E:H^{s}(\mathbb{R}^N \setminus \Omega) \to H^{s}(\mathbb{R}^N)$, without loss of generality, we can assume that $(u_n)$ is a bounded sequence in $H^{s}(\mathbb{R}^N)$, where we are identifying $u_n$ with $E(u_n)$.

In the sequel, for $\delta>0$ small enough, we introduce the sets 
$$
\Omega_\delta = \{x\in G:\;\;dist(x, \partial \Omega)\geq \delta\} \;\;\mbox{and}\;\;\hat{\Omega}_\delta = \{x\in G:\;\;dist(x, \partial \Omega)\leq \delta\}.
$$  
Then, by H\"older inequality,  
$$
\begin{aligned}
\int_{\hat{\Omega}_\delta}|u_n|^p &\leq |\hat{\Omega}_\delta|^{\frac{2_{s}^{*}-p}{2_{s}^{*}}}\left(\int_{\hat{\Omega}_\delta}|u_n|^{2_{s}^{*}}dx\right)^{\frac{p}{2_{s}^{*}}}\\
&\leq |\tilde{\Omega}_\delta|^{\frac{2_{s}^{*}-p}{2_{s}^{*}}} \|u_n\|_{L^{2_{s}^{*}}(G)}^{p} \leq |\hat{\Omega}_\delta|^{\frac{2_{s}^{*}-p}{2_{s}^{*}}} C\|u_n\|_{H_{G}^{s}}^{p}.
\end{aligned}
$$
As $(u_n)$ is bounded, given $\epsilon >0$, there is $\delta >0$ such that 
\begin{equation}\label{GS03}
\int_{\tilde{\Omega}_\delta}|u_n|^{p}dx < \frac{\epsilon}{2},\;\;\forall n \in \N.
\end{equation} 
On the other hand, let $\varphi \in C^\infty(\R^N)$ such that $0\leq \varphi (x)\leq 1$ for all $x\in \R^N$ and 
$$
\begin{aligned}
&\varphi(x) = 1,\;\,x\in \Omega_\delta,\\
&\varphi (x) = 0,\;\;x\in \Omega\;\;\mbox{and}\\
&|\nabla \varphi(x)|\leq M.
\end{aligned}
$$ 
Setting  $v_n(x) = \varphi (x)u_n(x)$, we can use the same arguments employed in \cite[Lemma 5.3]{EDNGPEV} to get  
$$
\|v_n\|_{s} \leq C\|u_n\|_{H_{G}^{s}}, \quad \forall n \in \mathbb{N},
$$
from where it follows that $(v_n)$ is a bounded sequence in $H^s(\R^N)$. Moreover, as 
$$
\begin{aligned}
\int_{B(y,T)}|v_n|^2dx &=\int_{B(y,T)\cap G}|v_n|^2dx\leq \int_{U(y,T)}|u_n|^2dx,
\end{aligned}
$$
by (\ref{GS01}),  
\begin{equation}\label{GS04}
\lim_{n\to \infty}\sup_{y\in \R^N}\int_{B(y,T)}|v_n|^2dx = 0.
\end{equation}
Hence, by \cite[Lemma 2.2]{PFAQJT}, 
\begin{equation}\label{GS05}
v_n \to 0\;\;\mbox{in}\;\;L^p(\R^N),\;\;\mbox{for}\;\;p\in (2,2_{s}^{*}).
\end{equation} 
Thereby, by (\ref{GS05}), 
\begin{equation}\label{GS06}
\int_{\Omega_\delta}|u_n|^pdx = \int_{\Omega_\delta}|v_n|^pdx\leq \int_{\R^N}|v_n|^pdx\to 0\;\;\mbox{as}\;\;n\to +\infty.
\end{equation}
Combining (\ref{GS03}) and (\ref{GS06}), we get  
\begin{equation}\label{GS07}
\int_{G}|u_n|^pdx \to 0\;\;\mbox{as}\;\;n\to +\infty.
\end{equation}
\end{proof}

 Our next step is to prove that functional $I$ satisfies the mountain pass geometry. 
\begin{lemma} The functional $I$ verifies the mountain pass geometry

\end{lemma}		
\begin{proof}	
Since $p\in (1, \frac{N+2s}{N-2s})$ and $Q$ is bounded, by Sobolev embedding we get 
$$
I(u) \geq \frac{1}{2}\|u\|_{H_{\tilde{\Omega}}^{s}}^2 - \frac{\|Q\|_\infty}{p+1}\|u\|_{H_{\tilde{\Omega}}^{s}}^{p+1}.
$$
Therefore, for every $u\in {H_{\tilde{\Omega}}^{s}}$ such that $\|u\|_{H_{\tilde{\Omega}}^{s}} = \rho$, with $\rho = \left(\frac{p+1}{4\|Q\|_\infty} \right)^{\frac{1}{p-1}}$, we compute 
$$I(u) \geq \alpha := \frac{\rho^2}{4}>0$$
So, $I$ verifies the first geometry condition of mountain pass Theorem. On the other hand, taking $u\in {H_{\tilde{\Omega}}^{s}}\setminus \{0\}$ and $t\geq 0$ we have 
$$
I(tu) = \frac{t^2}{2}\|u\|_{H_{\tilde{\Omega}}^{s}}^2 - \frac{t^{p+1}}{p+1}\int_{\R^N \setminus \Omega} Q(x)|u|^{p+1}dx
$$
Since $p+1 >2$, $I(tu) \to -\infty$ as $t\to \infty$, then $I$ verifies the second geometry condition of mountain pass theorem.
\end{proof}
The last lemma permits to apply the mountain pass theorem without Palais-Smale condition found in \cite{MW} to find a sequence $(u_n) \subset H_{\tilde{\Omega}}^{s}$ such that 
\begin{equation}\label{G01}
I(u_n) \to c_1\;\;\mbox{and}\;\;I'(u_n)\to 0
\end{equation}
where
\begin{equation}\label{G02}
c_1 = \inf_{u\in H_{\tilde{\Omega}}^{s}\setminus \{0\}}\sup_{t\geq 0} I(tu).
\end{equation}
Furthermore, we also have 
\begin{equation} \label{c1}
c_1 = \inf_{u\in \mathcal{N}}I(u),
\end{equation}
where 
$$
\mathcal{N} = \{u\in H_{\tilde{\Omega}}^{s}\setminus \{0\}: \;I'(u)u=0\}. 
$$
Hereafter, we say that $u \in H_{\tilde{\Omega}}^{s}$ is a ground state solution for $(P)$ when
$$
I(u)=c_1 \quad \mbox{and} \quad I'(u)=0.
$$

The next result shows an important relation between the levels $c_1$ and $c_\infty$.
\begin{proposition}\label{Gres02}
Suppose that $(Q_1)$ holds. Then 
\begin{equation}\label{G03}
0<c_1<c_\infty.
\end{equation}
\end{proposition}
\begin{proof}
Let $\tilde{u}$ be a ground state solution of $(P_\infty)$ and define $u_n(x) = \tilde{u}(x-\sigma_n)$, with $\sigma_n=(n,0,....,0) \in \R^N$. By (\ref{G02}), 
\begin{equation}\label{G04}
c_1\leq \max_{t\geq 0}I(tu_n).
\end{equation}
 Now, for every $t>0$ consider the function 
$$
h (t) = \frac{t^2}{2}\|u_n\|_{H_{\tilde{\Omega}}^{s}}^{2} - \frac{t^{p+1}}{p+1}\int_{\R^N \setminus \Omega} Q(x)|u_n|^{p+1}dx.
$$ 
A simple computation implies $h (0) = 0$, $h (t) < 0$ for $t$ small enough, and $h (t)>0$ for $t$ large enough. Accordingly, there is a unique $\gamma_n \in (0, \infty)$ such that
\begin{equation}\label{G05}
h (\gamma_n) = I(\gamma_n u_n) = \max_{t\geq 0} I(tu_n). 
\end{equation}
Thus $h' (\gamma_n) = 0$, which is equivalent to 
\begin{equation}\label{G06}
\frac{1}{2}\iint_{\R^{2N}\setminus \Omega^2}\frac{|u_n(x) - u_n(y)|^2}{|x-y|^{N+2s}}dy dx +\int_{\R^N\setminus \Omega}|u_n|^2dx = \gamma_{n}^{p-1}\int_{\tilde{\Omega}}Q(x)|u_n|^{p+1}dx.
\end{equation}
From definition of $c_1$, 
\begin{equation}\label{G07}
\begin{aligned}
c_1&\leq I(\gamma_n u_n) \\
& = I_\infty (\gamma_n u_n) - \frac{\gamma_n^2}{2}\left(\frac{1}{2}\iint_{\Omega \times \Omega}\frac{|u_n(x) - u_n(y)|^2}{|x-y|^{N+2s}}dy dx + \int_{\Omega}|u_n|^2 dx \right)\\
&+\frac{\gamma_{n}^{p+1}}{p+1}\int_{\R^N \setminus \Omega}(\tilde{Q} - Q(x))|u_n|^{p+1}dx + \frac{\gamma_{n}^{p+1}}{p+1}\int_{\Omega}\tilde{Q}|u_n|^{p+1}dx\\
&= I_\infty (\gamma_n u_n) - \frac{t_n\gamma_{n}^2}{2} + \frac{\gamma_n^{p+1}}{p+1}\int_{\Omega}\tilde{Q}|u_n|^{p+1}dx + \frac{\gamma_n^{p+1}}{p+1}\int_{\R^N \setminus \Omega} (\tilde{Q} - Q(x))|u_n|^{p+1}dx,
\end{aligned}
\end{equation}
where
$$
t_n = \frac{1}{2}\iint_{\Omega \times \Omega}\frac{|u_n(x) - u_n(y)|^2}{|x-y|^{N+2s}}dy dx + \int_{\Omega}|u_n|^2 dx.
$$

Note that $(\gamma_n)$ is bounded, because by (\ref{G06}),   
$$
0 \leq \gamma_{n}^{p-1} =\frac{\|u_n\|_{H_{\tilde{\Omega}}^{s}}^2}{\int_{\R^N \setminus \Omega}Q(x) |u_n|^{p+1}dx} \to \frac{\|\tilde{u}\|_{s}^{2}}{ \int_{\R^N}\tilde{Q}|\tilde{u}|^{p+1}dx}\;\;\mbox{as}\;\;n\to \infty.
$$
Thereby,  up to a subsequence, $\gamma_n \to \gamma_0$. We claim that $\gamma_0 = 1$. In fact, making the change of variable $\tilde{x} = x-\sigma_n$ and $\tilde{y} = y-\sigma_n$, we ascertain 
$$
\begin{aligned}
\|u_n\|_{H_{\tilde{\Omega}}^{s}}^{2} 
&= \iint_{\R^N\times \R^N} \chi_{\R^{2N}\setminus \Omega^2}(x,y)\frac{|\tilde{u}(x-\sigma_n) - \tilde{u}(y-\sigma_n)|^2}{|x-y|^{N+2s}}dy dx + \int_{\R^N} \chi_{\R^N \setminus \Omega} (x)\tilde{u}^2(x-\sigma_n)dx\\
&=\iint_{\R^N \times \R^N} \chi_{\R^{2N}\setminus \Omega^2}(x+\sigma_n, y + \sigma_n)\frac{|\tilde{u}(x) - \tilde{u}(y)|^2}{|x-y|^{N+2s}}dy dx + \int_{\R^N} \chi_{\R^N \setminus \Omega}(x+\sigma_n)\tilde{u}^2(x)dx.
\end{aligned}
$$
Now, since $|\sigma_n|\to \infty$, 
$$
\begin{aligned}
\chi_{\R^{2N}\setminus \Omega^2}(x+\sigma_n, y+\sigma_n)\frac{|\tilde{u}(x) - \tilde{u}(y)|^2}{|x-y|^{N+2s}}  \to \frac{|\tilde{u}(x) - \tilde{u}(y)|^2}{|x-y|^{N+2s}}  \;\;\mbox{a.e.}\;\; (x,y)\in \R^N \times \R^N
\end{aligned}
$$
and 
$$
\chi_{\R^N \setminus \Omega}(x+\sigma_n)|\tilde{u}(x)|^2 \to |\tilde{u}(x)|^2\;\;\mbox{a.e.} \;\;x\in \R^N.
$$
Furthermore,
$$
\left| \chi_{\R^{2N}\setminus \Omega^2}(x+\sigma_n, y+\sigma_n)\frac{|\tilde{u}(x) - \tilde{u}(y)|^2}{|x-y|^{N+2s}} \right| \leq \frac{|\tilde{u}(x) - \tilde{u}(y)|^2}{|x-y|^{N+2s}} \in L^1(\R^N \times \R^N) 
$$
and
$$
\left| \chi_{\R^N \setminus \Omega}(x+\sigma_n)\tilde{u}^2(x)\right| \leq  |\tilde{u}(x)|^2 \in L^1(\R^N).
$$
So, by Lebesgue's theorem,
\begin{equation}\label{G08}
\|u_n\|_{H_{\tilde{\Omega}}^{s}}^{2} \to \|\tilde{u}\|_{s} \;\;\mbox{as}\;\: n \to \infty.
\end{equation}
On the other hand, 
$$
\begin{aligned}
\gamma_{n}^{p-1}\int_{\R^N \setminus \Omega}Q(x)|u_n|^{p+1}dx &= \gamma_{n}^{p-1}\int_{\R^N}\chi_{\R^N \setminus \Omega}(x)Q(x)|\tilde{u}(x-\sigma_n)|^{p+1}dx\\
& = \gamma_{n}^{p-1}\int_{\R^N}\chi_{\R^N \setminus \Omega} (x+ \sigma_n)Q(x+\sigma_n)|\tilde{u}(x)|^{p+1}dx.
\end{aligned}
$$ 
By condition $(Q_1)$, 
$$
\gamma_{n}^{p-1}\chi_{\R^N \setminus \Omega}(x+\sigma_n)Q(x+\sigma_n)|\tilde{u}(x)|^{p+1} \to \gamma_{0}^{p-1}\tilde{Q}|\tilde{u}(x)|^{p+1}\;\;\mbox{a.e.}\;\;x\in \R^N.
$$
Also, since $Q$ and $(\gamma_n)$ are bounded, 
$$
|\chi_{\R^N \setminus \Omega}(x+\sigma_n)Q(x+\sigma_n)|\tilde{u}(x)|^{p+1}| \leq \tilde{B}|\tilde{u}(x)|^{p+1}\in L^1(\R^N).
$$
Thus, by Lebesgue's theorem 
\begin{equation}\label{G09}
\gamma_{n}^{p-1}\int_{\R^N\setminus \Omega}Q(x)|u_n(x)|^{p+1}dx \to \gamma_{0}^{p-1}\int_{\R^N}\tilde{Q}|\tilde{u}(x)|^{p+1}dx.
\end{equation}
As $\tilde{u}$ is a solution of $(P_\infty)$, the limits (\ref{G08}) and (\ref{G09}) together with the uniqueness of the limit yield $\gamma_0=1$. Accordingly, by (\ref{G07}), 

\begin{equation}\label{G10}
c_1\leq I_\infty (\tilde{u}) - \frac{t_n\gamma_n^2}{2} + s_n = c_\infty - \frac{ t_n\gamma_n^2}{2} + s_n=c_\infty + t_n\left(-\frac{ \gamma_n^2}{2} + \frac{s_n}{t_n}\right),
\end{equation}
where 
$$
s_n = \frac{\gamma_{n}^{p+1}}{p+1}\left(\int_{\Omega}\tilde{Q}|u_n(x)|^{p+1} + \int_{\R^N \setminus \Omega}(\tilde{Q} - Q(x))|u_n(x)|^{p+1}dx \right).
$$
We claim that 
$$
\frac{s_n}{t_n} \to 0\;\;\mbox{as}\;\;n \to +\infty,
$$
which is enough to confirm that $c_1 < c_\infty$.

To verifies this assertion, first of all, note that $\tilde{u}\in H^s(\R^N)$ yields $t_n \to 0$ as $n \to +\infty$. On the other hand, according to  $(Q_1)$, $\tilde{Q} - Q(x) \leq 0$ for all $x\in \R^N \setminus \Omega$, then 
\begin{equation}\label{G11}
\begin{aligned}
s_n &= \frac{\gamma_n^{p+1}}{p+1}\left(\tilde{Q}\int_{\Omega}|u_n(x)|^{p+1}dx + \int_{\R^N\setminus \Omega}(\tilde{Q} - Q(x))|u_n(x)|^{p+1}dx\right)\\
&\leq \frac{\tilde{Q}\gamma_n^{p+1}}{p+1}\int_{\Omega}|u_n(x)|^{p+1}dx.
\end{aligned}
\end{equation}
This inequality combines with Sobolev embedding and (\ref{G11}) to give
$$
\begin{aligned}
\frac{s_n}{t_n} & \leq \frac{\frac{\tilde{Q}\gamma_n^{p+1}}{p+1}\int_{\Omega}|u_n(x)|^{p+1}dx}{\frac{1}{2}\iint_{\Omega \times \Omega}\frac{|u_n(x) - u_n(y)|^2}{|x-y|^{N+2s}}dy dx + \int_{\Omega}|u_n|^2dx} \\
&\leq \frac{C_{p+1}^{p+1}\tilde{Q}\gamma_{n}^{p+1}}{(p+1)} \left( \frac{1}{2}\iint_{\Omega \times \Omega}\frac{|u_n(x) - u_n(y)|^2}{|x-y|^{N+2s}}dy dx + \int_{\Omega}|u_n|^2dx\right)^{p-1}\\
& \leq \tilde{C}\gamma_{n}^{p+1}t_n^{\frac{p-1}{2}} \to 0\;\;\mbox{as}\;\;n \to +\infty, 
\end{aligned}
$$
as asserted.
\end{proof}

\noindent 
{\bf Proof of Theorem \ref{ground}.} From (\ref{G01}), there exists a sequence $(u_n) \subset H_{\tilde{\Omega}}^{s}$ such that 
$$
I(u_n)\to c_1\;\;\mbox{and}\;\;I'(u_n)\to 0\;\;\mbox{as}\;\;n\to \infty.
$$   
Since $(u_n)$ is bounded, for some subsequence, there exists $u\in H_{\tilde{\Omega}}^{s}$, such that $u_n \rightharpoonup u$ in $H_{\tilde{\Omega}}^{s}$ and  $I'(u)=0$. Now, we are going to show that $u\not=0$.  Supposing by contradiction that $u=0$. Since $c_1>0$, the Proposition \ref{GC-C} ensures the existence of $r,\beta>0$ and $(y_n) \subset  \tilde{\Omega}$ with $|y_n|\to \infty$ as $n\to \infty$ such that
$$
\int_{B_r(y_n) \cap \tilde{\Omega}}|u_n|^{2}\,dx \geq \beta, \quad \forall n \in \mathbb{N}.
$$
Then,  for each $T>0$ fixed, there is $n_0=n_0(T) \in \mathbb{N}$ such that  
$$
B(0,T) \subset \R^N \setminus (\Omega - y_n), \quad \forall n \geq n_0.
$$ 
Setting $v_n(x)=u_n(x+y_n)$, we have that for some subsequence, $(v_n)$ is bounded in  $H^{s}(B(0,T))$ for all $T>0$. Indeed, as $(u_n)$ is bounded in $H_{\tilde{\Omega}}^{s}$,  there exists a positive constant $M$ such that 
$$
\begin{aligned}
M & \geq  \iint_{\R^{2N} \setminus \Omega^2} \frac{|u_n(x) - u_n(y)|^2}{|x-y|^{N+2s}}dy dx + \int_{\R^N \setminus \Omega}|u_n|^2dx\\
&= \iint_{\R^{2N} \setminus (\Omega - y_n)^2 } \frac{|v_n(x) - v_n(y)|^2}{|x-y|^{N+2s}}dy dx + \int_{\R^N \setminus (\Omega - y_n)}|v_n|^2dx \\
&\geq \iint_{B(0,T)\times B(0,T)} \frac{|v_n(x) - v_n(y)|^2}{|x-y|^{N+2s}}dy dx + \int_{B(0,T)}|v_n|^2dx=\|v_n\|_{H^s(B(0,T))}^{2}.
\end{aligned}
$$  
From this, there is a subsequence of $(v_n)$, still denoted by itself, and $v\in H_{loc}^{s}(\mathbb{R}^N)$ such that 
$$
v_n \rightharpoonup  v \;\;\mbox{in $H^s(B(0,T))$, as $n\to \infty$}.
$$
Then, by the lower semicontinuity of the norm  
$$
\|v\|_{H^s(B(0,T))}\leq \liminf_{n\to \infty}\|v_n\|_{H^s(B(0,T))} \leq M,\quad \forall T>0,
$$ 
from where it follows that $v \in H^{s}(\mathbb{R}^N)$.  

Now, let $\psi \in H^{s}(\mathbb{R}^N)$ be a test function with bounded support. As $I'(u_{n}) = o_n(1)$, we have  
\begin{equation}\label{G19}
I'(u_n)\psi(.-y_{n}) = o_n(1).
\end{equation}
Hence, 
\begin{equation}\label{G20}
\begin{aligned}
&\frac{1}{2}\iint_{\R^{2N}\setminus (\Omega-y_n)^2} \frac{[v_{n}(x) - v_{n}(y)][\psi(x) - \psi(y)]}{|x-y|^{N+2s}}dy dx + \int_{\R^N \setminus (\Omega -y_{n})} v_{n}(x)\psi(x)dx \\
&\hspace{6cm}= \int_{\R^{N}\setminus (\Omega - y_{n})} Q(x + y_{n})|v_{n}(x)|^{p-1}v_{n}(x)\psi(x)dx.
\end{aligned}
\end{equation} 
By the weak convergence of $(v_n)$ to $v$ in $H^s(B(0,T))$, we discover   
\begin{equation}\label{G21}
\begin{aligned}
\frac{1}{2}\iint_{\R^{2N}\setminus (\Omega-y_n)^2} &\frac{[v_{n}(x) - v_{n}(y)][\psi(x) - \psi(y)]}{|x-y|^{N+2s}}dy dx + \int_{\R^N \setminus (\Omega -y_{n})} v_{n}(x)\psi(x)dx \\
&\to \frac{1}{2} \iint_{\R^{2N}} \frac{[v(x) - v(y)][\psi(x) - \psi(y)]}{|x-y|^{N+2s}}dydx + \int_{\R^N} v(x)\psi(x)dx.
\end{aligned}
\end{equation}
On the other hand, the limit $|y_n| \to +\infty$ implies 
$$
Q(x+y_{n}) \to \tilde{Q}\;\;\mbox{a.e. $x\in \R^N$, as}\;\;n \to +\infty,
$$
and so,
$$
Q(x+y_{n})|v_{n}(x)|^{p-1}v_{n}(x) \to \tilde{Q}|v(x)|^{p-1}v(x)\;\;\mbox{a.e. $x\in \R^N$, as $n\to +\infty$}.
$$
This limit combined with the boundedness of $(v_{n})$ in $L^{p+1}(\R^N \setminus \Omega)$ permit to apply  
\cite[Lemma 4.6]{Kavian} to get 
\begin{equation}\label{G22}
\int_{B(0,T)} Q(x+y_{n})|v_{n}(x)|^{p-1}v_{n}(x)\psi(x)dx \to \int_{B(0,T)}\tilde{Q}|v(x)|^{p-1}v(x)\psi(x)dx.
\end{equation}
Therefore, from (\ref{G21})-(\ref{G22}), 
$$
I'_\infty(v)\psi = 0.
$$
Now, by density, the last equality yields $v$ is a nontrivial solution of $(P_\infty)$. Invoking Fatou's lemma, we compute   
$$
\begin{aligned}
c_\infty &\leq I_\infty (v) - \frac{1}{2}I'_\infty(v)v = \left(\frac{1}{2} - \frac{1}{p+1}\right) \int_{\R^N}\tilde{Q}|v(x)|^{p+1}dx\\
&\leq \liminf_{n\to \infty} \left(\frac{1}{2} - \frac{1}{p+1}\right) \int_{\R^N \setminus (\Omega - y_n)} Q(x+y_n)|v_n(x)|^{p+1}dx\\
&= \liminf_{n\to \infty} \left(\frac{1}{2} - \frac{1}{p+1}\right)\int_{\R^N \setminus \Omega}Q(x)|u_n(x)|^{p+1}dx \\
&= \limsup_{n\to \infty} \left(I(u_n) - \frac{1}{2}I'(u_n)u_n \right) =c_1,
\end{aligned}
$$
contrary to Proposition \ref{Gres02}. This confirm $u \not=0$. As $I$ is an even functional, it follows from (\ref{c1}) that any ground state solution $u$ of $I$ does not change sing, otherwise it is possible to show that $I(u) \geq 2c_1 $, which is absurd. This shows the existence of a non-negative ground state solution for $(P)$.

\section{Proof of Theorem \ref{nodal}}

In this section we are going to prove Theorem \ref{nodal}. We introduce the nodal set 
$$
\mathcal{M} = \{u\in \mathcal{N}:\;u^{\pm}\not \equiv 0, \;\;I'(u)u^{+} = I'(u)u^{-} = 0\}
$$
and consider the following real number 
$$
c = \inf_{u\in \mathcal{M}}I(u).
$$
Let us point out that for all $u\in H_{\tilde{\Omega}}^{s}$, 
\begin{equation}\label{N01}
[ u ]_{H_{\tilde{\Omega}}^{s}}^{2} = [u^+]_{H_{\tilde{\Omega}}^{s}}^{2} + [u^-]_{H_{\tilde{\Omega}}^{s}}^{2}- \frac{1}{2}\iint_{\R^{2N}\setminus \Omega^2}\frac{u^+(x)u^-(y) + u^-(x)u^+(y)}{|x-y|^{N+2s}}dy dx,
\end{equation}
where
$$
[u]_{H_{\tilde{\Omega}}^{s}}^{2} = \frac{1}{2}\iint_{\R^{2N}\setminus \Omega^2}\frac{|u(x) - u(y)|^2}{|x-y|^{N+2s}}dy dx.
$$
So
\begin{equation}\label{N02}
I(u) = I(u^+) + I(u^-) - \frac{1}{4}\iint_{\R^{2N}\setminus \Omega^2}\frac{u^+(x)u^-(y) + u^-(x)u^+(y)}{|x-y|^{N+2s}}dy dx,
\end{equation}
\begin{equation}\label{N03}
I'(u)u^+ = I'(u^+)u^+ - \frac{1}{2}\iint_{\R^{2N}\setminus \Omega^2}\frac{u^+(x)u^-(y) + u^-(x)u^+(y)}{|x-y|^{N+2s}}dy dx
\end{equation}
and 
\begin{equation}\label{N04}
I'(u)u^- = I'(u^-)u^- - \frac{1}{2}\iint_{\R^{2N}\setminus \Omega^2}\frac{u^+(x)u^-(y) + u^-(x)u^+(y)}{|x-y|^{N+2s}}dy dx.
\end{equation}
Recalling that 
$$
\iint_{\R^{2N}\setminus \Omega^2}\frac{u^+(x)u^-(y) + u^-(x)u^+(y)}{|x-y|^{N+2s}}dy dx \leq 0,  \quad \forall u \in H_{\tilde{\Omega}}^{s},
$$
it follows that
\begin{equation}\label{N05}
I'(u^{\pm})u^{\pm} \leq 0, \quad \forall u \in \mathcal{M}.
\end{equation}
\begin{lemma}\label{Nlm01}
There exists $\rho >0$ such that 
\begin{enumerate}
\item[$(i)$] $\|u^{\pm}\|_{H_{\tilde{\Omega}}^{s}} \geq \rho$ for all $u \in \mathcal{M}$.
\item[$(ii)$] $I(u) >0 $ and $\|u\|_{H_{\tilde{\Omega}}^{s}} \geq \rho$ for all $u\in \mathcal{N}$.
\item[$(iii)$] $c = \displaystyle \inf_{u\in \mathcal{M}}I(u)>0$.
\end{enumerate}
\end{lemma}
\begin{proof}
$(i)$ From (\ref{N05}),  
$$
\frac{1}{2}\|u^{\pm}\|_{H_{\tilde{\Omega}}^{s}}^{2} \leq \int_{\R^N \setminus \Omega}Q(x)|u^{\pm}|^{p+1}dx
 \leq \|Q\|_{\infty}\|u^{\pm}\|_{H_{\tilde{\Omega}}^{s}}^{p+1}, \quad \forall u \in \mathcal{M}.
$$ 
Thus, there exists $\rho>0$ such that 
$$
\|u^{\pm}\|_{H_{\tilde{\Omega}}^{s}}\geq \rho >0.
$$

\noindent 
$(ii)$  From $(i)$, $\|u\|_{H_{\tilde{\Omega}}^{s}}\geq \rho >0$  for any $u\in \mathcal{N}$. Hence by equality $I'(u)u=0$,
\begin{equation} \label{cpositivo}
I(u) = I(u) - \frac{1}{p+1}I'(u)u = \frac{1}{2}\left(\frac{1}{2} - \frac{1}{p+1} \right){\|u\|_{H_{\tilde{\Omega}}^{s}}^{2}} \geq \left(\frac{1}{2}-\frac{1}{p+1} \right)\frac{\rho^2}{2}>0.
\end{equation}
$(iii)$ An immediate consequence of (\ref{cpositivo})
\end{proof}
\begin{lemma}\label{Nlm02}
Suppose that $(Q_1)$ and $(Q_2)$ hold. Then 
\begin{equation}\label{N06}
0<c < c_1+c_\infty,
\end{equation}
for $R$ given in $(Q_2)$ large enough. 
\end{lemma}
\begin{proof}
Let $\tilde{u}$ be a ground state solution of $(P_\infty)$ and $u_1$ be a ground state solution of $(P)$. For $\alpha, \tau>0$ we define $w_\sigma(x) = \alpha u_1(x) - \tau \tilde{u}_\sigma(x)$, where $\tilde{u}_\sigma(x) = \tilde{u}(x-\sigma)$ and $\sigma=\sigma_R$ was given in $(Q_2)$.  Using the function  $w_\sigma$, we set the functions 
\begin{equation}\label{N07}
h_\sigma^{\pm}(\alpha, \tau) = I'(w_{\sigma})w_{\sigma}^{\pm}. 
\end{equation}
As $I'(u_1)u_1 = 0$, then
$$
I'\left(\frac{u_1}{2}\right)\frac{u_1}{2} >0 \quad \mbox{and} \quad I'(2u_1)(2u_1) < 0.
$$
We claim that 
\begin{equation}\label{N10}
I'\left(\frac{\tilde{u}_\sigma}{2}\right)\frac{\tilde{u}_\sigma}{2}>0,\;\;\mbox{for $|\sigma|$ large enough}.
\end{equation}
In fact, note that 
$$
\begin{aligned}
I'\left(\frac{\tilde{u}_\sigma}{2}\right)\frac{\tilde{u}_\sigma}{2} = \mathcal{A} - \mathcal{B},
\end{aligned}
$$ 
where
$$
\begin{aligned}
&\mathcal{A} = \frac{1}{2}\iint_{\R^{2N}} \frac{|(\tilde{u}_\sigma/2)(x) - (\tilde{u}_\sigma/2)(y)|^2}{|x-y|^{N+2s}}dy dx + \int_{\R^N} \left|\frac{\tilde{u}_\sigma}{2}\right|^2dx - \int_{\R^N}Q(x) \left|\frac{\tilde{u}_\sigma}{2}\right|^{p+1}dx,\;\;\mbox{and}\\
&\mathcal{B} = \frac{1}{2}\iint_{\Omega \times \Omega} \frac{|(\tilde{u}_\sigma/2)(x) - (\tilde{u}_\sigma/2)(y)|^2}{|x-y|^{N+2s}}dy dx + \int_{\Omega}\left| \frac{\tilde{u}_\sigma}{2}\right|^2dx - \int_{\Omega}Q(x)\left|\frac{\tilde{u}_\sigma}{2}\right|^{p+1}dx
\end{aligned}
$$
By $(Q_1)$ and Lebesgue's theorem, 
$$
\mathcal{A} \to I'_\infty \left(\frac{\tilde{u}}{2}\right)\frac{\tilde{u}}{2} >0\;\;\quad \mbox{and} \quad \mathcal{B} \to 0 \quad  \mbox{as}\;\;|\sigma| \to +\infty,
$$
proving (\ref{N10}). Now we claim that there exists $\sigma_0>0$ large enough such that 
\begin{equation}\label{N11}
h_\sigma^+({1}/{2}, \tau)>0 \quad \mbox{and} \quad h_\sigma^+(2, \tau)<0, \quad \mbox{for} \quad  |\sigma|\geq \sigma_0 \quad \mbox{and} \quad \tau \in [{1}/{2}, 2].
\end{equation}
In the same way, 
\begin{equation}\label{N12}
h_\sigma^-(\alpha, {1}/{2})>0 \quad \mbox{and} \quad h_\sigma^-(\alpha, 2)<0, \quad \mbox{for} \quad  |\sigma|\geq \sigma_0 \quad \mbox{and} \quad \alpha \in [{1}/{2}, 2].
\end{equation}
In fact, note that 
$$
\begin{aligned}
h_{\sigma}({1}/{2}, \tau) &= I'(w_\sigma)w_{\sigma}^{+} \\
&=\frac{1}{2}\iint_{\R^{2N}\setminus \Omega^2} \frac{[w_\sigma(x)-w_\sigma(y)][w_\sigma^+(x) - w_\sigma^+(y)]}{|x-y|^{N+2s}}dy dx + \int_{\R^N \setminus \Omega} w_\sigma(x)w_\sigma^+(x)dx \\
&- \int_{\R^N \setminus \Omega}Q(x)|w_\sigma|^{p-1}w_\sigma w_\sigma^+dx\\
&= \frac{1}{2}\iint_{\R^{2N}\setminus \Omega^2} \frac{[\frac{1}{2}u_1(x) - \frac{1}{2}u_1(y)][w_{\sigma}^{+}(x) - w_{\sigma}^{+}(y)]}{|x-y|^{N+2s}}dy dx\\ 
&- \frac{1}{2} \iint_{\R^{2N}\setminus \Omega^2} \frac{[\tau\tilde{u}(x-\sigma) - \tau\tilde{u}(y-\sigma)][w_{\sigma}^{+}(x) - w_{\sigma}^{+}(y)]}{|x-y|^{N+2s}}dy dx\\
&+ \int_{\R^N \setminus \Omega} w_\sigma(x)w_\sigma^+(x)dx - \int_{\R^N \setminus \Omega}Q(x)|w_\sigma|^{p-1}w_\sigma w_\sigma^+dx.
\end{aligned}
$$
Let 
$$
\begin{aligned}
	\Pi_1&= \frac{1}{2}\iint_{\R^{2N}\setminus \Omega^2} \frac{[\frac{1}{2}u_1(x)-\frac{1}{2}u_1(y)][w_\sigma^+(x) - w_\sigma^+(y)]}{|x-y|^{N+2s}}dy dx\\
	\Pi_2& = \frac{1}{2}\iint_{\R^{2N}\setminus \Omega^2} \frac{[\tau \tilde{u}(x-\sigma)-\tau\tilde{u}(y-\sigma)][w_\sigma^+(x) - w_\sigma^+(y)]}{|x-y|^{N+2s}}dy dx.
\end{aligned}
$$
Since 
$$
w_\sigma (x) = \frac{1}{2}u_1(x) - \tau \tilde{u}(x-\sigma) \;\;\mbox{and}\;\;w_{\sigma}^{+}(x) = \frac{1}{2}[|w_\sigma(x)| + w_\sigma(x)],
$$
we find
\begin{equation}\label{N13}
w_\sigma^+(x) \to \frac{1}{2}u_1(x)\;\;\mbox{a.e.}\;\,x\in \R^N\setminus \Omega.
\end{equation}	
As $(w_\sigma)$ is bounded in $H_{\tilde{\Omega}}^{s}$, the above limit ensures   
$$
w_\sigma^+  \rightharpoonup \frac{1}{2}u_1\;\;\mbox{in}\;\;H_{\tilde{\Omega}}^{s}, \quad \mbox{as} \quad |\sigma| \to +\infty,
$$
which implies 
\begin{equation}\label{N14}
\Pi_1 \to \frac{1}{2}\iint_{\R^{2N}\setminus \Omega^2} \frac{|\frac{1}{2}u_1(x) - \frac{1}{2}u_1(y)|^2}{|x-y|^{N+2s}}dy dx\;\;\mbox{as}\;\;|\sigma|\to +\infty.
\end{equation}
Setting $h_\sigma (x) = \frac{1}{2}u_1(x+\sigma) - \tau \tilde{u}(x)$ and making the change of variable $\tilde{x} = x-\sigma$ and $\tilde{y} = y-\sigma$, we discover 
	$$
	\Pi_2 = \frac{1}{2}\iint_{\R^{2N}\setminus (\Omega - \sigma)^2} \frac{[\tau \tilde{u}(x)-\tau\tilde{u}(y)][h_\sigma^+(x) - h_\sigma^+(y)]}{|x-y|^{N+2s}}dy dx.
	$$
	Consider 
	$$
	\mathbb{B}(0,T) = \{(x,y)\in \R^{2N}:\;\;|(x,y)|< T\},\;\;\mbox{for $T>0$}.
	$$
	Choose $|\sigma|$ large enough such that 
	$$
	\mathbb{B}(0,T) \subset \R^{2N}\setminus (\Omega - \sigma)^2.
	$$
Consequently
	$$
	\R^{2N} \setminus (\Omega - \sigma)^2 = \mathbb{B}(0,T)\cup \R^{2N}\setminus [\mathbb{B}(0,T) \cup (\Omega - \sigma)^2]
	$$
from where it follows that 
	$$
	\Pi_2 = \Pi_{2}^{1} + \Pi_{2}^{2},
	$$
with
	$$
	\Pi_{2}^{1}= \frac{1}{2}\iint_{\mathbb{B}(0,T)} \frac{[\tau \tilde{u}(x)-\tau\tilde{u}(y)][h_\sigma^+(x) - h_\sigma^+(y)]}{|x-y|^{N+2s}}dy dx
	$$
	and
	$$
	\Pi_{2}^{2}=\frac{1}{2}\iint_{\R^{2N}\setminus [\mathbb{B}(0,T) \cup (\Omega - \sigma)^2]} \frac{[\tau \tilde{u}(x)-\tau\tilde{u}(y)][h_\sigma^+(x) - h_\sigma^+(y)]}{|x-y|^{N+2s}}dy dx.
	$$
	Since $\tilde{u} \in H^s(\R^N)$, 
	\begin{equation}\label{N15}
	\begin{aligned}
	\Pi_{2}^{2}&=\iint_{\R^{2N}\setminus [\mathbb{B}(0,T) \cup (\Omega - \sigma)^2]} \frac{[\tau \tilde{u}(x)-\tau\tilde{u}(y)]}{|x-y|^{\frac{N}{2}+2s}}\frac{[h_\sigma^+(x) - h_\sigma^+(y)]}{|x-y|^{\frac{N}{2}+s}}dy dx\\
	&\leq \left( \iint_{\R^{2N}\setminus [\mathbb{B}(0,T) \cup (\Omega -\sigma)^2]} \frac{|\tau \tilde{u}(x) - \tau \tilde{u}(y)|^2}{|x-y|^{N+2s}} dy dx\right)^{1/2} \left(\iint_{\R^{2N}\setminus [\mathbb{B}(0,T) \cup (\Omega - \sigma)^2]} \frac{|h_{\sigma}^{+}(x) - h_{\sigma}^{+}(y)|^2}{|x-y|^{N+2s}}dy dx \right)^{1/2}\\
	&\leq C \left( \iint_{\R^{2N}\setminus \mathbb{B}(0,T)} \frac{|\tau \tilde{u}(x) - \tau \tilde{u}(y)|^2}{|x-y|^{N+2s}} dy dx\right)^{1/2} \\
	& \to 0\;\;\mbox{as $T\to +\infty$}
	\end{aligned}
	\end{equation}
Reasoning as in the proof (\ref{N14}),  
\begin{equation}\label{N16}
\Pi_{2}^{1}= \frac{1}{2}\iint_{\mathbb{B}(0,T)} \frac{[\tau \tilde{u}(x)-\tau\tilde{u}(y)][h_\sigma^+(x) - h_\sigma^+(y)]}{|x-y|^{N+2s}}dy dx \to 0\;\;\mbox{as}\;\;|\sigma|\to \infty.
\end{equation}
From  (\ref{N15}) and (\ref{N16}),  
\begin{equation}\label{N17}
\Pi_2 \to 0\;\;\mbox{as}\;\;|\sigma| \to \infty.
\end{equation}
Now note that 
$$
\int_{\R^N \setminus \Omega} w_\sigma(x)w_{\sigma}^+(x)dx = \int_{\R^N}\chi_{\Lambda_\sigma}(x)(w_\sigma^+)^2(x)dx,
$$ 
where $\Lambda_\sigma = \{x\in \R^N\setminus \Omega:\;\,\frac{1}{2}u_1(x) - \tau \tilde{u}(x-\sigma)>0\}$. Since $\tilde{u}(x-\sigma) \to 0$ a.e. in $\R^N$ as $|\sigma|\to \infty$ and $u_1>0$, we have 
$$
\chi_{\Lambda_\sigma}(x)\to 1\;\;\mbox{a.e. in $\R^N \setminus \Omega$}.
$$
Furthermore 
$$
\chi_{\Lambda_\sigma}(x)(w_\sigma^+)^2(x)\leq (\frac{1}{2}u_1(x))^2 \in L^1(\R^N \setminus \Omega)
$$
and
$$ 
 \chi_{\Lambda_\sigma}(x) (w_{\sigma}^+)^{2}(x) \to (\frac{1}{2}u_1(x))^2\;\;\mbox{a.e. in } \quad \R^N\setminus \Omega.
$$
Then by Lebesgue's Theorem, 
\begin{equation}\label{N18}
\int_{\R^N \setminus \Omega} w_\sigma(x)w_{\sigma}^{+}(x)dx \to \int_{\R^N \setminus \Omega} (\frac{1}{2}u_1(x))^2dx\;\;\mbox{as $|\sigma|\to \infty$}.
\end{equation} 
In the same way, we can show that 
\begin{equation}\label{N19}
\int_{\R^N \setminus \Omega}Q(x)|w_\sigma(x)|^{p-1}w_\sigma(x)w_\sigma^+(x)dx \to \int_{\R^N \setminus \Omega} Q(x)|\frac{1}{2}u_1(x)|^{p+1}dx.
\end{equation}
From (\ref{N14}) and (\ref{N17})-(\ref{N19}), 
 $$
 h_{\sigma}^{+}({1}/{2},\tau) \to I'({u_1}/{2})\frac{1}{2}u_1>0\;\;\mbox{as $|\sigma|\to \infty$}.
 $$
A similar argument shows that $h_\sigma^+(2,\tau)<0$ for $\sigma$ large and $\tau \in [{1}/{2},2]$ and (\ref{N12}). 

The estimates (\ref{N11}) and (\ref{N12}) permit to employ the Mean Value Theorem due to Miranda \cite{CM} to find $\alpha^*, \tau^*\in [{1}/{2}, 2]$, which depend on $\sigma$ such that 
$$
h_{\sigma}^{\pm}(\alpha^*, \tau^*) = 0\;\;\mbox{for any}\;\;|\sigma|\geq \sigma_0.
$$
Thus,
$$
\alpha^*u_1 - \tau^*\tilde{u}_\sigma\in \mathcal{M},\;\;\mbox{for}\;\;|\sigma|\geq \sigma_0.
$$
We would like to emphasize that the last inequality holds for $|\sigma|$ large, then by $(Q_2)$, it is enough to consider $R$ large enough to have $|\sigma|$ large enough.  

By the definition of $c$, it suffices to show that 
\begin{equation}\label{N25}
\sup_{\alpha, \tau \in [{1}/{2}, 2]}I(\alpha u_1 - \tau \tilde{u}_\sigma) < c_1 + c_\infty\;\;\mbox{for}\;\;|\sigma|\geq \sigma_0.
\end{equation}

Indeed, as $u_1$ is a non negative ground state solution of problem $(P)$, 
$$
\frac{1}{2}\iint_{\R^{2N}\setminus \Omega^2} \frac{[u_1(x) - u_1(y)][\tilde{u}_\sigma(x) - \tilde{u}_\sigma (y)]}{|x-y|^{N+2s}}dy dx + \int_{\R^N \setminus \Omega} u_1(x)\tilde{u}_\sigma (x)dx \geq 0. 
$$  
So 
$$
\begin{aligned}
\frac{1}{2}\iint_{\R^{2N}\setminus \Omega^2}& \frac{|w_\sigma(x) -w_\sigma (y)|^2}{|x-y|^{N+2s}}dy dx + \int_{\R^N \setminus \Omega} w_\sigma^2(x)dx\\
 &\leq \frac{1}{2}\iint_{\R^{2N}\setminus \Omega^2} \frac{|(\alpha u_1)(x) - (\alpha u_1)(y)|^2}{|x-y|^{N+2s}}dy dx + \int_{\R^N\setminus \Omega} (\alpha u_1)^2(x)dx \\
&+ \frac{1}{2}\iint_{\R^{2N}\setminus \Omega^2} \frac{|(\tau \tilde{u}_\sigma)(x) - (\tau \tilde{u}_\sigma)(y)|^2}{|x-y|^{N+2s}}dy dx + \int_{\R^N \setminus \Omega} (\tau \tilde{u}_\sigma)^2(x)dx.\\
\end{aligned}
$$
Furthermore, since for all $t,s \geq 0$ 
$$
|t-s|^{p+1} \geq t^{p+1} + s^{p+1} - C(t^ps + ts^p)
$$ 
for some positive constant $C$, we discover 
$$
\begin{aligned}
I(w_\sigma) 
&\leq \frac{1}{2}\left(\frac{1}{2}\iint_{\R^{2N}\setminus \Omega^2} \frac{|(\alpha u_1)(x) - (\alpha u_1)(y)|^2}{|x-y|^{N+2s}}dy dx + \int_{\R^N\setminus \Omega} (\alpha u_1)^2(x)dx\right)\\
&+\frac{1}{2}\left(\frac{1}{2}\iint_{\R^{2N}\setminus \Omega^2} \frac{|(\tau \tilde{u}_\sigma)(x) - (\tau\tilde{u}_\sigma)(y)|^2}{|x-y|^{N+2s}}dy dx + \int_{\R^N\setminus \Omega} (\tau\tilde{u}_\sigma)^2(x)dx\right)\\
&- \frac{1}{p+1}\int_{\R^N\setminus \Omega}\!\!\!\!\!Q(x) |\alpha u_1|^{p+1}dx - \frac{1}{p+1}\int_{\R^N\setminus \Omega}\!\!\!\!\!Q(x)|\tau \tilde{u}_\sigma|^{p+1} + \frac{C}{p+1}\int_{\R^N\setminus \Omega}\!\!\!\!\!Q(x)(u_1^p\tilde{u}_\sigma+ u_1\tilde{u}_{\sigma}^{p})dx\\
&= I(\alpha u_1) + I_\infty(\tau \tilde{u}) - \frac{1}{2^{p+1}(p+1)} \int_{\R^N \setminus \Omega} (Q(x) - \tilde{Q})|\tilde{u}_\sigma|^{p+1}dx + \frac{2^{p+1}}{p+1}\int_{\Omega} \tilde{Q}|\tilde{u}_\sigma(x)|^{p+1}dx\\
&+ \tilde{C}\int_{\R^N \setminus \Omega}[u_1^p\tilde{u}_\sigma + u_1\tilde{u}_{\sigma}^{p}]dx.
\end{aligned}
$$
Now, $(Q_2)$ combined with (\ref{decaimento}) ensures that 
\begin{equation}\label{N26}
\begin{aligned}
\int_{\R^N \setminus \Omega} (Q(x) - \tilde{Q})|\tilde{u}_\sigma(x)|^{p+1}dx &= \int_{\R^N \setminus (\Omega -\sigma)}(Q(x+\sigma) - \tilde{Q})|\tilde{u}(x)|^{p+1}dx\\
&\geq C \int_{B(0,2R)\setminus B(0,R)} \frac{R^\gamma}{|x|^{(N+2s)(p+1)}}dx. \\
\end{aligned}
\end{equation}
Since $R>D+1$,
\begin{equation}\label{N260}
\int_{\R^N \setminus \Omega} (Q(x) - \tilde{Q})|\tilde{u}_\sigma(x)|^{p+1}dx  \geq C_1 R^{N +\gamma - (N+2s)(p+1)},
\end{equation}  
for some $C_1>0$. By $(Q_2)$,    
$$
|x-\sigma| \geq |\sigma| - D \geq \frac{|\sigma|}{2} >D,
$$
and so, 
\begin{equation}\label{N27}
\begin{aligned}
\int_{\Omega} \tilde{Q} |\tilde{u}_\sigma(x)|^{p+1}dx & \leq C\int_{B(0,D)} \frac{1}{|x-\sigma|^{(N+2s)(p+1)}}dx\\
&\leq \frac{C|S^{N-1}|D^N}{D^{(N+2s)(p+1)}} = C_2,
\end{aligned}
\end{equation}
where $C_2$ does not depend on $R$. On the other hand, by H\"older inequality and Lemma \ref{Rlm01},  
\begin{equation}\label{N28}
\begin{aligned}
\int_{\R^N \setminus \Omega} u_1^p(x) \tilde{u}_\sigma(x)dx &\leq \left(\int_{\R^N \setminus \Omega} u_1^{p+1}(x)dx \right)^{\frac{p}{p+1}}\left(\int_{\R^N \setminus \Omega} \tilde{u}^{p+1}(x-\sigma)dx \right)^{\frac{1}{p+1}}\\
&\leq \left(\int_{\R^N \setminus \Omega} {\|u_1\|_{L^{\infty}(\R^N)}^{p-1}|u_1|^{2}}dx \right)^{\frac{p}{p+1}}\left(\int_{\R^N} \tilde{u}^{p+1}(x)dx \right)^{\frac{1}{p+1}}\\
& \leq \|u_1\|_{L^{\infty}(\R^N \setminus \Omega)}^{\frac{p(p-1)}{p+1}}\|u_1\|^{\frac{2p}{p+1}}_{L^{2}(\R^N) \setminus \Omega}\|\tilde{u}\|_{L^{p+1}(\R^N) 
} \leq C_3,\,\, 
\end{aligned}
\end{equation}
where $C_3$ does not depend on $R$, see Corollary \ref{KS}. 

In the same way,
\begin{equation}\label{N29}
\begin{aligned}
\int_{\R^N \setminus \Omega} u_1(x) \tilde{u}_\sigma^p(x)dx &\leq  \left(\int_{\R^N \setminus \Omega} u_{1}^{p+1}(x)\,dx \right)^{\frac{1}{p+1}}\left( \int_{\R^N \setminus \Omega} \tilde{u}_{\sigma}^{p+1}(x)dx\right)^{\frac{p}{p+1}}\\
&\leq  \|u_1\|^{\frac{p-1}{p+1}}_{L^{\infty}(\R^N \setminus \Omega)} \|u_1\|^{\frac{2}{p+1}}_{L^{2}(\R^N \setminus \Omega)} \|\tilde{u}\|^{p}_{L^{p+1}(\R^N)} \leq C_4,\\
\end{aligned}
\end{equation}
where $C_4$ does not depend on $R$.

According to  (\ref{N26})-(\ref{N29}), 
$$
\begin{aligned}
I(w_\sigma)  &\leq I(\alpha u_1) + I_\infty (\tau \tilde{u}) - \tilde{C}_1R^{N+\gamma - (N+2s)(p+1)} + C_D\\
&=I(\alpha u_1) + I_\infty (\tau \tilde{u}) + R^{N+\gamma - (N+2s)(p+1)}\left[- \tilde{C}_1 + \frac{C_D}{R^{N+\gamma - (N+2s)(p+1)}}\right],
\end{aligned}
$$
where $C_D$ is a constant that only depends on $D$. 
Hence, fixing $R_0$ large enough in the last inequality of way that the term inside of the brackets is negative, we deduce that 
\begin{equation}\label{N30}
\sup_{\alpha, \tau \in [{1}/{2}, 2]} I(\alpha u_1 - \tau \tilde{u}_\sigma) < c_1 + c_\infty,
\end{equation}  
from where it follows that $c < c_1 + c_\infty.$ \end{proof}

In order to get new estimates involving the level $c$, we will work with the following auxiliary problem  
$$
\left\{
\begin{aligned}
(-\Delta)^su + u &= \varphi_\varrho (x)Q(x)|u|^{p-1}u,\;\;\mbox{in}\:\:\R^N\setminus \Omega\\
\mathcal{N}_su(x) &= 0\;\;\mbox{in}\;\;{\Omega}.
\end{aligned}
\right.
\eqno{(P_\varrho)}
$$
where $\varphi_\varrho(x) = \varphi (\frac{x}{\varrho})$, $\varrho\gg D+1$, and  $ \varphi \in C^{\infty}_0(\mathbb{R}^N)$ satisfies 
$$
\varphi (x) = \begin{cases}
1,&|x|\leq 1\\
0,&|x| \geq 2.
\end{cases}
$$

Associated to problem $(P_\varrho)$ we have the energy functional  
$$
I_\varrho(u) = \frac{1}{4}\iint_{\R^{2N}\setminus \Omega^2}\frac{|u(x) - u(y)|^2}{|x-y|^{N+2s}}dy dx +\frac{1}{2}\int_{\R^N\setminus \Omega}|u|^2dx - \frac{1}{p+1}\int_{\R^N \setminus \Omega} Q(x)\varphi_\varrho(x)|u(x)|^{p+1}dx.
$$
Associated with $I_\varrho$ we have the nodal set
$$
\mathcal{M}_\varrho = \{u\in \mathcal{N}_\varrho:\;\;u^{\pm} \not \equiv 0\;\;\mbox{and}\;\;I'_{\varrho}(u)u^{\pm} = 0\}
$$
with
$$
\mathcal{N}_{\varrho} = \{u\in H_{\tilde{\Omega}}^{s}\setminus \{0\}:\;\;I'_\varrho(u)u=0\}
$$
and the number
$$
c_\varrho = \inf_{u\in \mathcal{M}_\varrho} I_\varrho(u).
$$
By similar reasoning as used in \cite{KTKWRW}( see also \cite{AlvesSouto}), we can show that for each $\varrho \gg D+1$ there exists $u_\varrho \in \mathcal{M}_\varrho$ such that $u_{\varrho}^{\pm} \not \equiv 0$ and $I_\varrho(u_\varrho) = c_\varrho >0$.

The next lemma establishes an important relation involving the levels $c_\varrho$ and $c$. 

\begin{lemma}\label{Nlm03}
$$
\lim_{\varrho\to +\infty}c_\varrho = c = \inf_{u\in \mathcal{M}}I(u).
$$
\end{lemma}
\begin{proof}
Fixed $u \in \mathcal{M}_\varrho$ and arguing as in \cite[Lemma 3.1]{KTKWRW}, there exist $t_\varrho, s_\varrho>0$ such that 
$$
t_\varrho u^+ + s_\varrho u^{-} \in \mathcal{M}.
$$
Since $I(u) \leq I_\varrho(u)$ for all $u\in H_{\tilde{\Omega}}^{s}$, it follows that 
$$
c\leq I(t_\varrho u^+ + s_\varrho u^{-}) \leq I_\varrho (t_\varrho u^+ + s_\varrho u^{-}) \leq I_\varrho(u) \leq c_\varrho, \quad \forall \varrho \gg D+1.
$$
Therefore 
\begin{equation}\label{N31}
c\leq \liminf_{\varrho\to +\infty}c_\varrho.
\end{equation}
On the other hand, given $w\in \mathcal{M}$, there exist $t_\varrho, s_\varrho>0$ such that 
$$
t_\varrho w^+ + s_\varrho w^-\in \mathcal{M}_\varrho.
$$ 
A direct computation gives that $(t_\varrho)$ and $(s_\varrho)$ are bounded, otherwise we have the limit below  
$$
\lim_{\varrho \to +\infty}I_\varrho(t_\varrho w^+ + s_\varrho w^-) <0,
$$
which is impossible, because  $I_\varrho(t_\varrho w^+ + s_\varrho w^-) \geq c_\varrho >0$ for all $\varrho>0$. Hence, by Lebesgue's theorem   
$$
\int_{\R^N \setminus \Omega} (1-\varphi_\varrho (x))Q(x)|t_\varrho w^+ + s_\varrho w^-|^{p+1}dx \to 0\;\;\mbox{as}\;\;\varrho\to \infty,
$$ 
from where it follows that 
$$
c_\varrho \leq I(w) + o_\varrho(1).
$$
From this,
$$
\limsup_{\varrho\to +\infty}c_\varrho\leq I(w),\;\;\forall w\in \mathcal{M},
$$
which leads to  
\begin{equation}\label{N37}
\limsup_{\varrho\to +\infty}c_\varrho \leq c.
\end{equation}
Combining (\ref{N31}) and (\ref{N37}), we discover
$$
\lim_{\varrho\to \infty}c_\varrho = c.
$$
\end{proof}

\noindent 
{\bf Proof of Theorem \ref{nodal}:} In what follows, we set $\varrho_n \to +\infty$ and $u_n=u_{\varrho_n}$. Since 
$$
c + \|u_n\|_{H_{\tilde{\Omega}}^{s}}\geq I_{\varrho_n}(u_n) - \frac{1}{p+1}I'_{\varrho_n}(u_n)u_n \geq  \frac{1}{2}\left( \frac{1}{2} - \frac{1}{p+1}\right)\|u_n\|_{H_{\tilde{\Omega}}^{s}}^{2}, 
$$
we conclude that $(u_n)$ is bounded in $H_{\tilde{\Omega}}^{s}$. Then, for some subsequence, there is $u\in H_{\tilde{\Omega}}^{s}$ such that 
$$
u_n  \rightharpoonup u\;\;\mbox{in}\;\;H_{\tilde{\Omega}}^{s} \quad \mbox{and}\quad I'(u)=0. 
$$ 
Now we are going to show that $u^{\pm} \neq 0$. Keeping this in mind, we need to consider three cases:
\begin{itemize}
\item[(i)] $u^+ = u^-=0$.
\item[(ii)] $u^+\neq 0$ and $u^- = 0$.
\item[(iii)] $u^+ = 0$ and $u^- \neq 0$
\end{itemize}
We are going to prove that the cases above do not hold, which permits to conclude that $u^{\pm} \neq 0$. However, it is enough to prove only $(i)$, because the other cases follow with the same type of arguments.

Since Proposition \ref{GC-C} cannot be applied in this case, there exist $\eta, \kappa>0$ and sequences $(y_{n}^{1})$ and $(y_{n}^{2})$ in $\R^N\setminus \Omega$ with $|y_n^1|, |y_{n}^{2}|\to \infty$ such that    
\begin{equation}\label{N38}
\liminf_{n\to +\infty} \int_{U(y_{n}^{1}, \kappa)} |u_{n}^{+}|^2dx\geq \eta\quad \mbox{and}\quad \liminf_{n \to +\infty} \int_{U(y_{n}^{2}, \kappa)} |u_{n}^{-}|^2dx\geq \eta.
\end{equation}
Letting $w_n (x) = u_{n}(x+y_{n}^{1})$, $z_{n}(x) = u_n(x+ y_{n}^{2})$ and arguing as in the proof of Theorem \ref{ground},  there exist $w,z\in H^s(\R^N)\setminus \{0\}$ such that $w_n  \rightharpoonup w$ and $z_n  \rightharpoonup z$ in $H^s(B_T(0))$ for all $T>0$,  with $w^+\neq 0$ and $z^-\neq 0$. Now, let $\psi \in H^{s}(\mathbb{R}^N)$ be a test function with bounded support. Since $I'_{\varrho_n}(u_{n}) = 0$, then 
\begin{equation}\label{N39}
I'_{\varrho_n}(u_n)\psi(.-y_{n}^{1}) = 0,
\end{equation}
for $n$ large enough. From (\ref{N39}),  
\begin{equation}\label{N40}
\begin{aligned}
&\frac{1}{2}\iint_{\R^{2N}\setminus (\Omega-y_n^1)^2} \frac{[w_{n}(x) - w_{n}(y)][\psi(x) - \psi(y)]}{|x-y|^{N+2s}}dy dx + \int_{\R^N \setminus (\Omega - y_{n}^{1})} w_{n}\psi dx \\
&\hspace{6cm}= \int_{\R^{N}\setminus (\Omega -y_{n}^{1})} \varphi_{\varrho_n}(x+y_{n}^{1})Q(x + y_{n}^{1})|w_{n}|^{p-1}w_{n}\psi dx.
\end{aligned}
\end{equation} 
Fixing $T>0$ of way that $supp \psi \subset B(0,T)$, the weak convergence of $(w_n)$ to $w$ in $H^{s}(B_T(0))$ gives
\begin{equation}\label{N41}
\begin{aligned}
\frac{1}{2}\iint_{\R^{2N}\setminus (\Omega-y_n^{1})^2 } &\frac{[w_{n}(x) - w_{n}(y)][\psi(x) - \psi(y)]}{|x-y|^{N+2s}}dy dx + \int_{\R^N \setminus (\Omega - y_{n}^{1})} w_{n}(x)\psi(x)dx \\
&\to \frac{1}{2} \iint_{\R^{2N}} \frac{[w(x) - w(y)][\psi(x) - \psi(y)]}{|x-y|^{N+2s}}dydx + \int_{\R^N} w(x)\psi(x)dx
\end{aligned}
\end{equation}
On the other hand, as $x/\varrho_n \to 0$ for all $x \in \mathbb{R}^N$, we can assume that there is $A \in [0,1]$, independent of $x$, such that
$$
\varphi_{\varrho_n}(x+y_{n}^{1}) \to A \;\;\mbox{and}\;\;Q(x+y_{n}^{1}) \to \tilde{Q}\;\;\mbox{a.e. $x\in \R^N$, as}\;\;n \to +\infty.
$$
Consequently
$$
\varphi_{\varrho_n}(x+ y_{n}^{1})Q(x+y_{n}^{1})|w_{n}(x)|^{p-1}w_{n}(x) \to A\tilde{Q}|w(x)|^{p-1}w(x)\;\;\mbox{a.e. $x\in \R^N$, as $n \to +\infty$}.
$$
Moreover, by boundedness of $(w_{n})$ in $L^{p+1}(\R^N \setminus \Omega)$, 
$$
\int_{\R^N\setminus \Omega}\left| \varphi_{\varrho_n}(x+y_{n}^{1})Q(x+y_{n}^{1})|w_{n}(x)|^{p-1}w_{n}(x)\right|^{\frac{p+1}{p}}dx \leq \|Q\|_{\infty}\int_{\R^N \setminus \Omega} |w_{n}(x)|^{p+1}dx.
$$
Thereby, by \cite[Lemma 4.6]{Kavian}, 
\begin{equation}\label{N42}
\int_{\R^{N}\setminus (\Omega + y_{n}^{1})} \varphi_{\varrho_n}(x+y_{n}^{1})Q(x+y_{n}^{1})|w_{n}(x)|^{p-1}w_{n}(x)\psi(x)dx \to \int_{\R^N}A\tilde{Q}|w(x)|^{p-1}w(x)\psi(x)dx,
\end{equation}
Now, according to (\ref{N41})-(\ref{N42}), 
$$
\frac{1}{2} \iint_{\R^{2N}} \frac{[w(x) - w(y)][\psi(x) - \psi(y)]}{|x-y|^{N+2s}}dydx + \int_{\R^N} w(x)\psi(x)dx=\int_{\R^N}A\tilde{Q}|w(x)|^{p-1}w(x)\psi(x)dx
$$
for all $\psi \in H^{s}(\mathbb{R}^N)$. Hence, setting the functional $I_{A,\infty}: H^s(\R^N) \to \R$ by  
\begin{equation}\label{P01}
I_{A, \infty} (u) = \frac{1}{4}\int_{\R^N}\int_{\R^N} \frac{|u(x) - u(y)|^2}{|x-y|^{N+2s}}dy dx +\frac{1}{2} \int_{\R^N}|u|^2dx - \frac{1}{p+1}\int_{\R^N}A\tilde{Q}|u(x)|^{p+1}dx,
\end{equation} 
we deduce that
$$
I'_{A,\infty}(w)\psi=0. 
$$
Since $w \not= 0$, taking $\psi=w^{+}$ we compute   
$$
I'_{A,\infty}(w^{+})w^{+} \leq 0.
$$
Then, there is $t_{w} \in (0,1]$ such that 
$$
t_{w}w^{+} \in \mathcal{N}_{A,\infty}= \{u\in H^{s}(\mathbb{R}^N)\setminus \{0\}:\;\;I'_{A,\infty}(u)u=0\}.
$$ 
Considering 
$$
c_{A,\infty}=\inf_{u\in \mathcal{N}_{A,\infty}} I_{A,\infty}(u),
$$
we derive that 
$$
c_\infty \leq c_{A,\infty},
$$
and so,
$$
c_\infty\leq c_{A,\infty}\leq I_{A,\infty}(t_ww^{+})= I_{A,\infty}(t_ww^{+})- \frac{1}{2}I'_\infty(t_ww^+)t_ww^+,
$$
that is,
$$
c_\infty\leq \left(\frac{1}{2}-\frac{1}{p+1} \right)\int_{\R^N}A\tilde{Q}|w^+(x)|^{p+1}dx \leq  \left(\frac{1}{2}-\frac{1}{p+1} \right)\int_{\R^N}\tilde{Q}|w^+(x)|^{p+1}dx.
$$
A similar arguments yields 
$$
c_\infty\leq \left(\frac{1}{2}-\frac{1}{p+1} \right)\int_{\R^N}A\tilde{Q}|z^-(x)|^{p+1}dx \leq \left(\frac{1}{2}-\frac{1}{p+1} \right)\int_{\R^N}\tilde{Q}|z^-(x)|^{p+1}dx.
$$
Hence, 
$$
2c_\infty \leq \left(\frac{1}{2}-\frac{1}{p+1} \right)\int_{\R^N}\tilde{Q}|w^+(x)|^{p+1}dx+\left(\frac{1}{2}-\frac{1}{p+1} \right)\int_{\R^N}\tilde{Q}|z^-(x)|^{p+1}dx.
$$
Combining the  Fatou's Lemma with the Lemmas \ref{Nlm02} and \ref{Nlm03}, we find
$$
\begin{aligned}
2c_{\infty} &\leq  \liminf_{n \to +\infty} \left(\frac{1}{2}-\frac{1}{p+1}\right)\left[ \int_{\R^N\setminus (\Omega - y_{n}^{1})} Q(x+y_{n}^{1})|w_{n}^{+}(x)|^{p+1} + \int_{\R^N\setminus (\Omega - y_{n}^{2})}Q(x+y_{n}^{2})|z_n^{-}(x)|^{p+1}dx\right] \\
&=\liminf_{n \to +\infty} \left(\frac{1}{2}-\frac{1}{p+1}\right)\int_{\R^N\setminus \Omega}Q(x)|u_n(x)|^{p+1}dx\\
& = \liminf_{n \to +\infty} \left[I_{\varrho_n}(u_n) - \frac{1}{2}I'_{\varrho_n}(u_n)u_n\right] = \lim_{n \to +\infty} I_{\varrho_n}(u_n)= \lim_{n \to +\infty} c_{\varrho_n}=c < c_1+c_\infty,
\end{aligned}
$$
which is absurd. 

\section{Appendix: Some properties of the ground state solution of (P)}

In this section, our main goal is to study the $L^\infty$ estimate and decay at infinite of the ground state solution $u_1$ of ($P$) that was obtained in Section 3. We start our analysis with the following lemma. 
\begin{lemma}\label{Rlm01}
The ground state solution  $u_1$ belongs to $L^\infty(\R^N\setminus \Omega)$.
\end{lemma}
\begin{proof}  In this proof we adapt for our case some arguments found in \cite[Lemma 5.4]{AlvesAmbrosio}. In what follows we denote $u_1$ by $u$. Moreover, for all $t\in \R$ and $L>0$, we set
\begin{equation}\label{A01}
t_L = \mbox{sgn}(t) \min\{|t|, L\}.
\end{equation}
By \cite[Lemma 3.1]{AISMMS}, for all $a,b\in \R$, $\beta > 1$ and $L>0$ we have 
\begin{equation}\label{A02}
(a-b)(a|a|_{L}^{2(\beta -1)} - b|b|_{L}^{2(\beta -1)}) \geq \frac{2\beta -1}{\beta^2} (a|a|_{L}^{\beta - 1} - b|b|_{L}^{\beta -1})^2.
\end{equation}
Since the mapping 
$t\to t|t|_{L}^{2(\beta-1)} \;\,\mbox{is Lipschitz in $\R$},$  
then $uu_{L}^{2(\beta-1)} \in H_{\tilde{\Omega}}^{s}$. Taking $v = uu_{L}^{2(\beta-1)}$ as a test function in $(P)$, the Lemma \ref{Pnta01} together with (\ref{A02}) and the boundedness of $Q$ leads to 
$$
\begin{aligned}
&\|uu_{L}^{\beta-1}\|_{L^{2_{s}^{*}}(\R^N \setminus \Omega)}^{2} \leq C \|uu_{L}^{\beta-1}\|_{H_{\tilde{\Omega}}^{s}}^{2}\\
& = C\left(\iint_{\R^{2N}\setminus \Omega^2} \frac{|(uu_{L}^{\beta-1})(x) - (uu_{L}^{\beta-1})(y)|^2}{|x-y|^{N+2s}}dy dx + \int_{\R^N \setminus \Omega} (uu_{L}^{\beta-1})^2dx\right)\\
&\leq C\left( \frac{\beta^2}{2\beta-1} \iint_{\R^{2N}\setminus \Omega^2} \frac{(u(x) - u(y))(u(x)u_{L}^{2(\beta-1)}(x) - u(y)u_{L}^{2(\beta-1)}(y))}{|x-y|^{N+2s}} dy dx + \int_{\R^N\setminus \Omega} u^2u_{L}^{2(\beta-1)}dx\right)\\
&\leq \frac{C\beta^2}{2\beta-1} \left(\frac{1}{2}\iint_{\R^{2N}\setminus \Omega^2} \frac{(u(x) - u(y))(u(x)u_{L}^{2(\beta-1)}(x) - u(y)u_{L}^{2(\beta-1)}(y))}{|x-y|^{N+2s}} dy dx + \int_{\R^N\setminus \Omega} u^2u_{L}^{2(\beta-1)}dx\right)\\
&\leq C\beta^2 \int_{\R^{N}\setminus \Omega} Qu^{p+1}u_{L}^{2(\beta - 1)}dx\\
&\leq \tilde{C}\beta^2 \int_{\R^N\setminus \Omega} u^{p+1}u_{L}^{2(\beta - 1)}dx
\end{aligned}
$$
Let $w_{L} = uu_{L}^{\beta -1}$. By H\"older inequality,  
$$
\|w_{L}\|_{L^{2_{s}^{*}}(\R^N \setminus \Omega)}^{2} \leq \tilde{C}\beta^2 \left( \int_{\R^N \setminus \Omega} u^{2_{s}^{*}}dx\right)^{\frac{p-1}{2_{s}^{*}}}\left( \int_{\R^N \setminus \Omega} w_{L}^{\alpha_{s}^{*}}dx\right)^{\frac{2}{\alpha_{s}^{*}}},
$$
where $\alpha_{s}^{*} = \frac{22_{s}^{*}}{2_{s}^{*} - (p-1)} \in (2, 2_{s}^{*})$. Since $u\in H_{\tilde{\Omega}}^{s}$, 
\begin{equation}\label{A03}
\|w_{L}\|_{L^{2_s^*}(\R^N \setminus \Omega)}^{2} \leq \tilde{C}_1\beta^2 \|w_L\|_{L^{\alpha_{s}^{*}}(\R^N \setminus \Omega)}^{2}.
\end{equation}
Now, note that if $u^{\beta}\in L^{\alpha^{*}_{s}}(\R^{N}\setminus \Omega)$, from the definition of $w_{L}$, and by using the fact that $u_{L}\leq u$ and  (\ref{A03}), we derive
\begin{equation}\label{A04}
\|w_{L}\|_{L^{2_{s}^{*}}(\R^{N}\setminus \Omega)}^{2}\leq \tilde{C}_1 \beta^{2}  \left(\int_{\R^{N}\setminus \Omega} (uu_{L}^{\beta-1})^{\alpha_{s}^{*}} dx\right)^{\frac{2}{\alpha^{*}_{s}}}\leq \tilde{C}_1\beta^2 \left( \int_{\R^N \setminus \Omega} u^{\beta \alpha_{s}^{*}} dx\right)^{\frac{2}{\alpha_{s}^{*}}}<\infty.
\end{equation}
Taking the limit as $L \rightarrow +\infty$ in (\ref{A04}) and employing the Fatou's Lemma, we find 
\begin{equation}\label{A05}
\|u\|_{L^{\beta 2_{s}^{*}}(\R^{N}\setminus \Omega)}\leq \tilde{C}_1^{\frac{1}{\beta}} \beta^{\frac{1}{\beta}} \|u\|_{L^{\beta \alpha^{*}_{s}}(\R^{N}\setminus \Omega)}
\end{equation}
whenever $u^{\beta \alpha^{*}}\in L^{1}(\R^{N}\setminus \Omega)$. Now, choosing $\beta:=\frac{2^{*}_{s}}{\alpha^{*}_{s}}>1$  and observing that $u\in L^{2^{*}_{s}}(\R^{N}\setminus \Omega)$, it is easy to check that inequality above holds for this choice of $\beta$. Then, by using the fact that $\beta^{2}\alpha^{*}_{s}=\beta 2^{*}_{s}$, it follows that \eqref{A05} holds with $\beta$ replaced by $\beta^{2}$.
Therefore, 
$$
\|u\|_{L^{\beta^{2} 2^{*}_{s}}(\R^{N}\setminus \Omega)}\leq \tilde{C}_{1}^{\frac{1}{\beta^{2}}} \beta^{\frac{2}{\beta^{2}}} \|u\|_{L^{\beta^{2} \alpha^{*}_{s}}(\R^{N}\setminus \Omega)}\leq  \tilde{C}_{1}^{\left(\frac{1}{\beta}+\frac{1}{\beta^{2}}\right)} \beta^{\frac{1}{\beta}+\frac{2}{\beta^{2}}} \|u\|_{L^{\beta \alpha^{*}_{s}}(\R^{N}\setminus \Omega)}.
$$
Iterating this process, and recalling that $\beta \alpha^{*}_{s} = 2^{*}_{s}$, we can infer that for every $m\in \mathbb{N}$
\begin{equation}\label{A06}
\|u\|_{L^{\beta^{m} 2^{*}_{s}}(\R^{N}\setminus \Omega)}\leq \tilde{C}_1^{\sum_{j=1}^{m}\frac{1}{\beta^{j}}} \beta^{\sum_{j=1}^{m} j\beta^{-j}} \|u\|_{L^{2^{*}_{s}}(\R^{N}\setminus \Omega)}.
\end{equation}
Now, taking the limit in (\ref{A06}) as $m \rightarrow +\infty$, we get  
$$
\|u\|_{L^{\infty}(\R^{N}\setminus \Omega)}\leq \tilde{C}_1^{\sigma_{1}} \beta^{\sigma_{2}}K,
$$
where 
\begin{equation} \label{K}
K=\|u\|_{L^{2^{*}_{s}}(\R^{N}\setminus \Omega)}, \quad \sigma_{1}:=\sum_{j=1}^{\infty}\frac{1}{\beta^{j}}<\infty \quad \mbox{ and } \quad \sigma_{2}:=\sum_{j=1}^{\infty}\frac{j}{\beta^{j}}<\infty.
\end{equation}
\end{proof}

\begin{corollary} \label{KS} The constant $K$ given in (\ref{K}) does not depend on $R$. Hence, $\|u_1\|_{L^{\infty}(\R^{N}\setminus \Omega)}\leq M_2$, for some constant $M_2$ that is independent of $R>D+1$. Moreover, $\|u_1\|_{L^{2}(\R^{N}\setminus \Omega)}$ is also bounded from above by a constant that does not depend on $R>D+1$.
	
\end{corollary}
\begin{proof} Let $\varphi \in C^{\infty}_{0}(B(0,D+1) \setminus \Omega)$. From (\ref{G02}),
$$
c_1 \leq \max_{t\geq 0}I(t\varphi)\leq \left(\frac{1}{2}-\frac{1}{p+1}\right)\frac{\|\varphi\|^{\frac{p+1}{p-1}}_{H_{\tilde{\Omega}}^{s}}}{\left(\int_{\R^N \setminus \Omega} \tilde{Q}|\varphi|^{p+1}dx\right)^{\frac{2}{p-1}}}=A,
$$
where $A$ does not depend on $R>D+1$. As $u_1$ is a solution, it follows that
$$
c_1=I(u_1)-\frac{1}{p+1}I'(u_1)u_1 \geq \frac{1}{2}\left(\frac{1}{2}-\frac{1}{p+1}\right)\|u_1\|^{2}_{H_{\tilde{\Omega}}^{s}}.
$$
Therefore, 
$$
\|u_1\|_{H_{\tilde{\Omega}}^{s}} \leq \sqrt{\frac{4(p+1)A}{(p-1)}}= C_*,
$$
where $C_*$ does not depend on $R$. This inequality implies 
$$
\|u_1\|_{L^{2}(\R^{N}\setminus \Omega)} \leq \sqrt{\frac{4(p+1)A}{(p-1)}}= C_*.
$$
Now, by Sobolev embedding, there is $C>0$, independent of $R>D+1$ such that
$$
\|u_1\|_{L^{2^*_s}(\mathbb{R}^N \setminus \Omega)} \leq C \|u_1\|_{H_{\tilde{\Omega}}^{s}}.
$$ 
Then
$$
\|u_1\|_{L^{2^*_s}(\mathbb{R}^N \setminus \Omega)} \leq M_1, 
$$ 
for some $M_1$ that is independent of $R>D+1$. This shows the desired result. 
\end{proof}

As a byproduct of the last lemma we have the following corollary

\begin{corollary} The ground state solution $u_1$ is a bounded function in $\mathbb{R}^N$, that is, $u_1 \in L^{\infty}(\mathbb{R}^N)$ and $\|u_1\|_{L^{\infty}(\R^{N})}\leq M_2$, where $M_2$ was given in Corollary \ref{KS}.

\end{corollary}
\begin{proof} By (\ref{I02}),
$$
\mathcal{N}_su_1(x) = C_{N,s} \int_{\R^N \setminus \Omega} \frac{u_1(x) - u_1(y)}{|x-y|^{N+2s}}dy,\;\;x\in {\Omega}.	
$$
Since $u_1$ is a solution of $(P)$, we have $\mathcal{N}_su_1(x)=0$ for all $ x\in \Omega$, hence
$$
u_1(x)=\frac{\int_{\R^N \setminus \Omega} \frac{u_1(y)}{|x-y|^{N+2s}}dy}{\int_{\R^N \setminus \Omega} \frac{1}{|x-y|^{N+2s}}dy}, \quad \forall x \in \Omega. 
$$
From this, 
$$
0 \leq u_1(x) \leq \|u_1\|_{L^{\infty}(\mathbb{R}^N \setminus \Omega)}, \quad \forall x \in \Omega.
$$
Recalling that $\partial \Omega$ has Lebesgue's measure zero, we can conclude that $u_1 \in L^{\infty}(\mathbb{R}^N)$ and $\|u_1\|_{L^{\infty}(\R^{N})}\leq M_2$.
\end{proof}

Our next goal is to show that $u_1(x) \to 0$ as $|x| \to +\infty$. Nevertheless, in order to prove this, we will firstly study some properties of the solution of the following linear problem. 
\begin{equation}\label{A07}
\left\{\begin{array}{l} 
(-\Delta)^sv + v = g(x)\;\;\mbox{in}\;\;\R^N,\\
v\in H^s(\R^N),
\end{array}
\right.
\end{equation}
where 
$$
g(x) = Q(x)\hat{u}^p(x)
$$ 
and 
$$
\hat{u}(x) = \begin{cases}
u_1(x),&x\in \R^N \setminus \Omega\\
0,&x\in \Omega.
\end{cases}
$$
Note that $g\in L^2(\R^N)$, because 
$$
\begin{aligned}
\int_{\R^N} |g(x)|^2dx &= \int_{\R^N} Q^2(x) \hat{u}^{2p}(x)dx \leq \|Q\|_{\infty} \|u_1\|_{L^\infty(\R^N \setminus \Omega)}^{2p-2}\int_{\R^N\setminus \Omega} |u_1|^2dx < +\infty.
\end{aligned}
$$
Consequently, by Riesz's Theorem,  problem (\ref{A07}) has a unique weak solution $v\in H^s(\R^N)$, which is given by   
\begin{equation}\label{A08}
v(x) = (\mathcal{K} * g)(x) = \int_{\R^N} \mathcal{K}(x-\xi)g(\xi)d\xi,
\end{equation} 
where $\mathcal{K}$ is the Bessel kernel
\begin{equation}\label{A09}
\mathcal{K}(x) = \mathcal{F}^{-1}\left( \frac{1}{1+|\xi|^{2\alpha}}\right)(x),
\end{equation}
The function $K$ satisfies the following properties: 
\begin{itemize}
\item[$(K_1)$] $\mathcal{K} $ is positive, radially symmetric and smooth in $\mathbb{R}^{N} \setminus \{0\}$, \\
\item[$(K_2)$] There is $C>0$ such that  
$$
\mathcal{K}(x) \leq \frac{C}{|x|^{N+2s}}, \quad \forall x \in \mathbb{R}^{N} \setminus \{0\} 
$$
\item[$(K_3)$] There is a constant $C$ such that 
$$
\nabla \mathcal{K}(x) \leq \frac{C}{|x|^{N+1+2s}}\;\;\mbox{if}\;\;|x|\geq 1.
$$ 
\item[$(K_4)$] $\mathcal{K} \in L^{q}(\mathbb{R}^{N}), \quad \forall q \in [1,N/N-2s)$. 
\end{itemize}
For more details see \cite{PFAQJT}. Since $u(x)\geq 0$ for all $x\in \R^N\setminus \Omega$, $u \not \equiv 0$ and  $\mathcal{K}$ is positive, then $v(x)>0$ for all $x \in \mathbb{R}^N$. 

By using the information above, we are able to prove the following result
\begin{lemma}\label{continuity}
 The function $v$ is continuous, that is, $v\in C(\R^N)$.
 \end{lemma}
\begin{proof}
Let $\delta >0$, $x_0\in \R^N$ and $T > |x_0|+2 \delta$. For any $x\in B(x_0, \delta)$, 
$$
\begin{aligned}
|v(x) - v(x_0)| &= \left|\int_{\R^N} \mathcal{K}(x-\xi)g(\xi)d\xi - \int_{\R^N} \mathcal{K}(x_0 - \xi)g(\xi)d\xi\right|\\
&\leq \int_{\R^N} |\mathcal{K}(x-\xi) - \mathcal{K}(x_0-\xi)||g(\xi)|d\xi \\
&= \int_{B(0,T)} |\mathcal{K}(x-\xi) - \mathcal{K}(x_0-\xi)||g(\xi)|d\xi + \int_{B^c(0,T)} |\mathcal{K}(x-\xi) - \mathcal{K}(x_0-\xi)||g(\xi)|d\xi.
\end{aligned}
$$
Note that, by H\"older inequality,  
$$
\begin{aligned}
\int_{B^c(0,T)} |\mathcal{K}(x-\xi) - \mathcal{K}(x_0-\xi)||g(\xi)|d\xi & \leq \left( \int_{B^c(0,T)} |\mathcal{K}(x-\xi)- \mathcal{K}(x_0-\xi)|^2d\xi\right)^{1/2} \left( \int_{B^c(0, T)} |g(\xi)|^2d\xi\right)^{1/2}.
\end{aligned}
$$ 
Since $\mathcal{K}$ is smooth, there exists $C>0$ 
$$
\begin{aligned}
|\mathcal{K}(x - \xi) - \mathcal{K}(x_0 -\xi)| &\leq |\nabla \mathcal{K} (x_0-\xi + \theta (x-x_0))||x-x_0|\\
&\leq C \frac{1}{|x_0-\xi + \theta(x-x_0)|^{N+1+2s}} |x-x_0|\\
&\leq C 2^{N+1+2s} \frac{|x-x_0|}{|\xi|^{N+1+2s}}.
\end{aligned}
$$ 
Then
$$
\begin{aligned}
\int_{B^c(0,T)} |\mathcal{K}(x-\xi) - \mathcal{K}(x_0 - \xi)|^2 d\xi &\leq \tilde{C} |x-x_0|^2 \int_{B^c(0,T)} \frac{d\xi}{|\xi|^{2(N+1+2s)}}\\
&\leq \tilde{C} \delta^2 \int_{T}^{+\infty} \frac{r^{N-1}}{r^{2(N+1+2s)}}dr\\
& = \tilde{C}\delta^2 \frac{1}{T^{2(N+1+2s) - N}}
\end{aligned}
$$
and
\begin{equation*}
\int_{B^c(0,T)} |\mathcal{K}(x-\xi) - \mathcal{K}(x_0-\xi)||g(\xi)|d\xi \leq  \tilde{C} \frac{\delta}{T^{\frac{N}{2}+ 1+2s}} \left( \int_{\R^N} |g(\xi)|^2d\xi\right)^{1/2}.
\end{equation*}
Therefore, given $\epsilon$, we can fix $\delta$ small enough such that
\begin{equation}\label{A10}
\int_{B^c(0,T)} |\mathcal{K}(x-\xi) - \mathcal{K}(x_0-\xi)||g(\xi)|d\xi < \frac{\epsilon}{3}.
\end{equation}
On the other hand, fixing $q \in (1,N/N-2s)$, $q'=\frac{q}{q-1}$ and using $(K_4)$, we obtain by  H\"older inequality  
\begin{equation*}
 \int_{B(x_0,\delta)} |\mathcal{K}(x-\xi) - \mathcal{K}(x_0-\xi)||g(\xi)|d\xi \leq C\left(\int_{B(x_0,\delta)} |g(\xi)|^{q'}\,d\xi\right)^{\frac{1}{q'}}.  
\end{equation*}
From this, we can fix $\delta>0$ small enough such that 
\begin{equation}\label{A11}
\int_{B(x_0,\delta)} |\mathcal{K}(x-\xi) - \mathcal{K}(x_0-\xi)||g(\xi)|d\xi < \frac{\epsilon}{3}.  
\end{equation}
Finally, as $f(x,\xi)=\mathcal{K}(x-\xi)-\mathcal{K}(x_0-\xi)$ is a continuous function in $Y=B(x_0,\delta/2) \times (B(0,T) \setminus B(x_0,\delta))$,  it follows that 
$$
\lim_{x \to x_0}\sup_{(x,\xi) \in Y}|\mathcal{K}(x-\xi) - \mathcal{K}(x_0-\xi)|=0,
$$
implying that
\begin{equation}\label{A111}
\int_{B(0,T) \setminus B(x_0,\delta)} |\mathcal{K}(x-\xi) - \mathcal{K}(x_0-\xi)||g(\xi)|d\xi < \frac{\epsilon}{3},  
\end{equation}
when $|x-x_0|$ is smaller enough. Now, the lemma follows from (\ref{A10})-(\ref{A111}).

\end{proof}

Our next lemma studies the behavior of $v$ at infinity. In this proof, we use some arguments developed in Alves and Miyagaki \cite[Lemma 2.6]{CAOM}.   
\begin{lemma} \label{ZERO}
$$	
v(x) \to 0 \quad \mbox{as} \quad |x| \to +\infty.
$$
\end{lemma}
\begin{proof}
Given $\delta >0$, consider the sets
$$
A_\delta=\{y \in \mathbb{R}^{N}\,:\,  |y-x|\geq 1/\delta  \}
$$
and
$$
B_\delta=\{y \in \mathbb{R}^{N}\,:\,  |y-x|<1/\delta  \}.
$$
Hence, 
$$
\begin{aligned}
0\leq v(x) &=  \int_{\mathbb{R}^{N}}\mathcal{K}(x-y)Q(y)\hat{u}(y)^{p}dy\\
& = \int_{A_\delta}\mathcal{K}(x-y)Q(y)\hat{u}^{p}dy+\int_{B_\delta}\mathcal{K}(x-y)Q(y)\hat{u}^{p}dy.
\end{aligned}
$$
From definition of $A_\delta$ and $(K_2)$, 
\begin{equation} \label{A12}
\begin{aligned}
\int_{A_\delta}\mathcal{K}(x-y)Q(y)\hat{u}^{p}dy &\leq \|Q\|_{\infty}\|u_1\|^p_\infty\int_{A_\delta}\mathcal{K}(x-y)dy \\
&\leq \|Q\|_{\infty}C\int_{|x-y|\geq \frac{1}{\delta}}\frac{dy}{|x-y|^{N+2s}}=C_1\delta^{2s}.
\end{aligned}
\end{equation}
Fixing $q \in (1,N/N-2s)$, $q'=\frac{q}{q-1}$ and using $(K_4)$, we obtain by  H\"older inequality   
$$
\begin{aligned}
\int_{B_\delta}\mathcal{K}(x-y)Q(y)\tilde{u}^{p}dy&\leq \int_{B_\delta} \mathcal{K}(x-y)Q(y)|u_1(y)|^pdy\\
&\leq K \left( \int_{B_\delta} \mathcal{K}^q(x-y)dx\right)^{1/q} \left( \int_{B_\delta} |u_1(y)|^{2q'}dy\right)^{1/q'}.
\end{aligned}
$$
As $u_1\in L^{2q'}(\R^N \setminus \Omega)$, 
$$
\|u_1\|_{L^{2q'}(B_\delta)} \to 0\;\;\mbox{as}\;\;|x|\to +\infty.
$$
Therefore, there are  $T>0$ such that
\begin{equation} \label{A13}
\int_{B_\delta}\mathcal{K}(x-y)Q(y)\hat{u}^{p}dy\leq \delta, \quad \forall  |x|\geq T.
\end{equation}
From (\ref{A12}) and (\ref{A13}),
\begin{equation} \label{A14}
\int_{\mathbb{R}^{N}}\mathcal{K}(x-y)Q(y)\hat{u}^{p} dy\leq C_1\delta^{2s}+\delta, \quad \forall |x|\geq T.
\end{equation}
Since $\delta$ is arbitrary, the proof is finished. 
\end{proof}

Next lemma establishes the behavior of $u_1$ at infinite.  
\begin{lemma}\label{decay}
$$
u_1(x)\to 0\;\;\mbox{as}\;\;|x|\to \infty.
$$
\end{lemma}
\begin{proof}
Let $v$ the positive solution of the linear problem (\ref{A07}) and $T>0$ such that 
$$
\Omega \subset B(0,T).
$$ 
Then there exists $C \gg 1$ such that 
$$
V(x) = Cv(x) \geq 1+\|u_1\|_{L^\infty(\R^N)}, \quad \mbox{for} \quad |x| \leq T.
$$
Moreover, $V$ is solution of the problem 
\begin{equation}\label{A15} 
(-\Delta)^s V + V = CQ(x)\hat{u}^p(x)\;\;\mbox{in}\;\;\R^N
\end{equation} 
and 
\begin{equation}\label{A16}
V(x) \to 0\;\;\mbox{as}\;\;|x| \to \infty.
\end{equation}
Let 
$$
\varphi (x) = \begin{cases}
(u_1-V)^+(x),&x\in B^c(0,T)\\
0,&x\in B(0,T).
\end{cases}
$$
We claim that 
\begin{equation}\label{A17}
\varphi \equiv 0.
\end{equation}
Assume for a moment that (\ref{A17}) is true. Then,  
$$
u_1(x)\leq V(x)\;\;\mbox{a.e.}\;\;x\in \R^N \setminus B(0,T).
$$
This combined with (\ref{A16}) gives  
\begin{equation}\label{A18}
u_1(x) \to 0\;\;\mbox{as}\;\;|x|\to \infty,
\end{equation}
as asserted.

\noindent 
{\bf Proof of the claim.} Since $\varphi \in H^{s}(\R^N)$ and $u_1$ is solution of problem $(P)$, 
\begin{equation}\label{A19}
\frac{1}{2}\iint_{\R^{2N} \setminus \Omega^2}\frac{[u_1(x)-u_1(y)][\varphi (x) - \varphi (y)]}{|x-y|^{N+2s}}dy dx + \int_{\R^N \setminus \Omega} u_1(x)\varphi (x)dx = \int_{\R^N \setminus \Omega} Q(x)u_1^p(x)\varphi (x)dx.
\end{equation}
Moreover as $V$ is solution of (\ref{A15}), 
$$
\frac{1}{2}\iint_{\R^{2N}} \frac{[V(x) - V(y)][\varphi (x) - \varphi (y)]}{|x-y|^{N+2s}}dy dx + \int_{\R^N} V(x)\varphi (x)dx = \int_{\R^N} CQ(x)\hat{u}^p(x)\varphi (x)dx.
$$
Since
$$
 \int_{\R^N} CQ(x)\hat{u}^p(x)\varphi (x)dx =\int_{\R^N \setminus \Omega} CQ(x)\hat{u}^p(x)\varphi (x)dx= \int_{\R^N \setminus \Omega} CQ(x)u_1^p(x)\varphi (x)dx,
$$
it follows that
\begin{equation}\label{A20}
\frac{1}{2}\iint_{\R^{2N}} \frac{[V(x) - V(y)][\varphi (x) - \varphi (y)]}{|x-y|^{N+2s}}dy dx + \int_{\R^N} V(x)\varphi (x)dx = \int_{\R^N \setminus \Omega} CQ(x)u_1^p(x)\varphi (x)dx.
\end{equation}
Now by subtracting (\ref{A19}) with (\ref{A20}), we find 
$$
\frac{1}{2}\iint_{\R^{2N}\setminus B^2(0,T)}\frac{[(u_1-V)(x) - (u_1-V)(y)][\varphi (x) - \varphi (y)]}{|x-y|^{N+2s}}dy dx + \int_{\R^N \setminus B(0,T)} (u_1-V)(x)\varphi (x)dx \leq 0.
$$
Using the fact that $V(x) \geq u_1(x)$ for $x \in B(0,T)$, it is easy to check that 
$$
[(u_1-V)(x) - (u_1-V)(y)][\varphi (x) - \varphi (y)]\geq 0, \quad (x,y) \in \R^{2N}\setminus B^2(0,T).
$$
Thus, as $\R^{N}\setminus B(0,T))^{2} \subset \R^{2N}\setminus B^2(0,T)$, we get
$$
\frac{1}{2}\iint_{(\R^{N}\setminus B(0,T))^{2}} \frac{|(u_1-V)^{+}(x) - (u_1-V)^+(y)|^2}{|x-y|^{N+2s}}dy dx + \int_{\R^N\setminus B(0,T)} [(u_1-V)^+]^2(x)dx \leq 0,
$$
implying that  $(u_1-V)^+ \equiv 0$. 
\end{proof}

\begin{lemma}\label{decay2}
There exists $C>0$, such that 
$$
0\leq u_1(x) \leq \frac{C}{|x|^{N+2s}}, \quad \forall  x \in \mathbb{R}^N \setminus \{0\}.
$$
\end{lemma}
\begin{proof}
Arguing as in \cite[Lemma 4.3]{PFAQJT}, it is possible to prove that there is a smooth positive function $w$ in $\R^N$ satisfying  
\begin{equation}\label{A21}
(-\Delta)^sw(x) + \frac{1}{2}w(x) \geq  0 \quad \mbox{for}\;\;|x|>T
\end{equation} 
in the classical sense, and 
\begin{equation}\label{A22}
0< w(x) \leq \frac{k_1}{|x|^{N+2s}}, \quad \forall x \in \mathbb{R}^N \setminus \{0\}.
\end{equation}
Note that (\ref{A21}) is equivalent to
\begin{equation}\label{A23}
\frac{1}{2}\iint_{\R^{2N}} \frac{[w(x)-w(y)][\phi(x)-\phi(y)]}{|x-y|^{N+2s}}dy dx + \frac{1}{2}\int_{\R^N} w(x)\phi(x)dx \geq  0,
\end{equation}
for all $\phi \in H^s(\R^N)$ with $\phi \geq 0$ and $supp \phi \subset B^{c}(0,T)$. Without loss of generality, as $\displaystyle \inf_{|x| \leq T}w(x)>0$, we can assume that 
$$
w(x) \geq 1+\|u_1\|_{L^{\infty}(\R^N)} \;\;\mbox{for}\;\;|x| \leq T.
$$
Note that, by (\ref{A18}), there is $T>0$ large enough such that 
\begin{equation}\label{A25}
u_1(x)\left( Q(x)u_1^{p-1}(x) - \frac{1}{2}\right) \leq 0,\;\;\mbox{for}\;\;|x|\geq T.
\end{equation} 
As in the last lemma, considering the function 
$$
\varphi (x) = \begin{cases}
(u_1-w)^+(x),&x\in \R^N \setminus B(0,T)\\
0,&x\in B(0,T).
\end{cases}
$$ 
it follows from  (\ref{A25}), 
\begin{equation}\label{A26}
\frac{1}{2}\iint_{\R^{2N}\setminus B^2(0,T)} \frac{[u_1(x) - u_1(y)][\varphi (x) -\varphi (y)]}{|x-y|^{N+2s}}dy dx + \frac{1}{2}\int_{\R^N \setminus B(0,T)} u_1(x)\varphi (x)dx\leq 0.
\end{equation} 
Therefore, from (\ref{A23}) and (\ref{A26}),   
\begin{equation}\label{A27}
\frac{1}{2}\iint_{\R^{2N}\setminus B^2(0,T)} \frac{[(u_1-w)(x) - (u_1-w)(y)][\varphi (x) - \varphi (y)]}{|x-y|^{N+2s}}dy dx + \frac{1}{2}\int_{\R^N \setminus B(0,T)} (u_1-w)(x)\varphi (x)dx \leq 0.
\end{equation}
Arguing as in Lemma \ref{decay}, 
$$
\frac{1}{2}\iint_{(\R^{N}\setminus B(0,T))^{2}} \frac{|(u_1-w)^+(x) - (u_1-w)^+(y)|^2}{|x-y|^{N+2s}}dy dx + \frac{1}{2}\int_{\R^N \setminus B(0,T)} [(u_1-w)^+(x)]^2dx \leq 0.
$$
that is, $(u_1-w)^+ \equiv 0$. Thereby,   
\begin{equation}\label{A28}
u_1(x) \leq w(x) \leq \frac{k_2}{|x|^{N+2s}}\;\;\mbox{for all}\;\;x\in \R^N \setminus B(0,T).
\end{equation}
Now, the result follows by using the fact that $u_1 \in L^{\infty}(\mathbb{R}^N)$.
\end{proof}

\noindent \textbf{Acknowledgment}: The authors thank the referee for his/her comments that were very important to improve the paper. C.O. Alves was partially supported by CNPq/Brazil 304804/2017-7 and C.E. Torres Ledesma was partially supported by INC Matem\'atica  88887.136371/2017.


\end{document}